\documentclass{article}
\usepackage[journal=JEMS,lang=british]{ems-journal-jems} 

\usepackage{amsfonts, xifthen, latexsym, amscd}
\usepackage{amsrefs,mdframed, etoolbox}
\usepackage{multirow}

\usepackage{tikz}
\usetikzlibrary{arrows}
\usepackage{subcaption}
\usepackage{listings}
\allowdisplaybreaks

\newcommand{\hide}[1]{}

\numberwithin{equation}{section}

\newtheorem{theorem}[equation]{Theorem} 
\newtheorem*{theorem*}{Theorem}

\newtheorem{lemma}[equation]{Lemma}
\newtheorem*{lemma*}{Lemma}

\newtheorem*{claim*}{Claim}
\newtheorem{proposition}[equation]{Proposition}

\newtheorem{cory}[equation]{Corollary}

\theoremstyle{definition}

\newtheorem{remark}[equation]{Remark}

\newenvironment{myrem}{%
	\begin{mdframed}[%
		linewidth=2pt, 
		topline=false, 
		bottomline=false, 
		rightline=false,%
		leftmargin=0pt, 
		innerleftmargin=0.4em, 
		rightmargin=0pt, 
		innerrightmargin=0pt, 
		innertopmargin=0pt ,%
		innerbottommargin=0pt, 
		splittopskip=\topskip, 
		splitbottomskip=0.3\topskip, %
		]%
		\begin{remark}
		}{%
		\end{remark}
	\end{mdframed}
}

\newcommand{\beq}[1][]{ 
	\ifthenelse{\isempty{#1}}{\begin{equation}}{\begin{equation}\label{#1}} 
		}
		\newcommand{\eeq}{\end{equation}}

	\makeatletter
	\newtheorem*{rep@theorem}{\rep@title}
	\newcommand{\newreptheorem}[2]{%
		\newenvironment{rep#1}[1]{%
			\def\rep@title{#2 \ref{##1}}%
			\begin{rep@theorem}}%
			{\end{rep@theorem}}}
	\makeatother
	\newreptheorem{theorem}{Theorem}
	\newreptheorem{lemma}{Lemma}

	\DeclareFontFamily{U}{mathx}{\hyphenchar\font45}
	\DeclareFontShape{U}{mathx}{m}{n}{<-> mathx10}{}
	\DeclareSymbolFont{mathx}{U}{mathx}{m}{n}
	
	\newcommand{\al}{{\alpha}}
	
	\newcommand{\be}{{\beta}}

	\newcommand{\eps}{{\varepsilon}}
	\newcommand{\de}{{\delta}}
	\newcommand{\De}{{\Delta}}

	\renewcommand{\phi}{{\varphi}}

	\renewcommand{\ge}{\geqslant}
	\renewcommand{\le}{\leqslant}

	\newcommand\N{\mathbb N}

	\newcommand\RR{\mathbb R}

	\newcommand\NN{\mathbb N}
	
	\renewcommand{\cal}[1]{{\mathcal #1}}
	
	\provideboolean{theno_parts}
	\setboolean{theno_parts}{true}
	\renewcommand{\part}[1][]{%
		\ifthenelse{\boolean{theno_parts}}{%
			\begin{enumerate}[wide,label=(\alph*),itemsep=1pt,topsep=0pt,#1]%
				\item
			}{%
				\item
			}
			\setboolean{theno_parts}{false}%
		}
		\newcommand{\trap}{%
			\ifthenelse{\boolean{theno_parts}}{}{%
			\end{enumerate}
		}
		\setboolean{theno_parts}{true}
	}

	\newcommand{\newop}[2]{%
		\expandafter\def\csname #1\endcsname{\operatorname{#2}}
	}
	
	\newcommand{\bits}{b}
	\newcommand\MT{\operatorname{M}}
	\newcommand\I{\operatorname{I}}
	\newcommand\B{\operatorname{B}}
	\newcommand\IB{\operatorname{IB}}
	
	\newcommand\Rel{\operatorname{Rel}}
	
	\newcommand\Var{\operatorname{Var}}
	\newcommand\Cl{\operatorname{Cla}}

	\newop{Pow}{Pow}
	\newop{im}{im}

	\newop{subexp}{SubExp}
	\newop{colsubexp}{ColSubExp}
	
	\newop{Pr}{\mathbb{P}}
	\newop{E}{\mathbb{E}}
	\newop{maxoutdeg}{max-outdeg}
	\newop{maxdeg}{max-deg}
	\newop{res}{res}

	\newcommand{\used}{\textsc{Used}}
	\newcommand{\unused}{\textsc{Unused}}
	\newop{varcount}{H}
	
	\newcommand{\Canvas}{\operatorname{Canvas}}
	\newcommand{\Trees}{\operatorname{Trees}}

	\newop{final}{\textsc{Final}}
	\newop{prev}{\textsc{Viol}}
	\newop{len}{len}
	\newop{For}{For}
	\newop{C}{\cal C}

	\newcommand{\me}{\mathrm{e}}
	\newcommand{\nn}{R}
	
\begin{document}

	\title{Moser--Tardos Algorithm with small number of random bits}
	\titlemark{Moser--Tardos Algorithm with small number of random bits}
	
	\emsauthor{1}{
		\givenname{Endre}
		\surname{Cs\'oka}
		\mrid{982650}
		\zblid{csoka.endre}
		\orcid{0000-0003-2945-6848}}{E.~Cs\'oka}
	\emsauthor{2}{
		\givenname{Łukasz}
		\surname{Grabowski}
		\mrid{993752}
		\zblid{grabowski.lukasz}
		\orcid{0000-0002-9470-4679}}{Ł.~Grabowski}
	\emsauthor{3}{
		\givenname{András}
		\surname{Máthé}
		\mrid{771705}
		\zblid{mathe.andras}
		\orcid{0000-0002-8898-8978}}{A.~Máthé}
	\emsauthor*{4}{
		\givenname{Oleg}
		\surname{Pikhurko}
		\mrid{334779}
		\zblid{pikhurko.oleg}
		\orcid{0000-0002-9657-4011}}{O.~Pikhurko}
	\emsauthor{5}{
		\givenname{Konstantinos}
		\surname{Tyros}
		\mrid{905104}
		\zblid{tyros.konstantinos}
		\orcid{0000-0002-8890-7074}}{K.~Tyros}
		
\Emsaffil{1}{
	\pretext{}
	\organisation{Alfr\'ed R\'enyi Institute of Mathematics}
	\zip{1053}
	\city{Budapest}
	\country{Hungary}
	\affemail{csokaendre@gmail.com}
	}		
\Emsaffil{2}{
	\pretext{}
	\department{Mathematics Institute}
	\organisation{Leipzig University}
	\zip{D-04009}
	\city{Leipzig}
	\country{Germany}
	\affemail{lukasz.grabowski@math.uni-leipzig.de}
}		
\Emsaffil{3}{
	\pretext{}
	\department{Mathematics Institute}
	\organisation{University of Warwick}
	\zip{CV4 7AL}
	\city{Coventry}
	\country{UK}
	\affemail{a.mathe@warwick.ac.uk}
}		
\Emsaffil{4}{
	\pretext{}
	\department{Mathematics Institute and DIMAP}
	\organisation{University of Warwick}
	\zip{CV4 7AL}
	\city{Coventry}
	\country{UK}
	\affemail{o.pikhurko@warwick.ac.uk}
}	
\Emsaffil{5}{
	\pretext{}
	\department{Department of Mathematics}
	\organisation{\\ University of Athens}
	\zip{157 84}
	\city{Athens}
	\country{Greece}
	\affemail{ktyros@math.uoa.gr}
}		
	
\classification[03E15, 05C85]{68W20}

\hide{
05D40  Probabilistic methods in extremal combinatorics, including polynomial methods
28A05  Classes of sets (Borel fields, $\sigma$-rings, etc.), measurable sets, Suslin sets, analytic sets 
68W20  Randomized algorithms
68W40  Analysis of algorithms
05C85  Graph algorithms (graph-theoretic aspects)
05C15  Coloring of graphs and hypergraphs
03E15  Descriptive set theory
}


\begin{abstract}
	We study a variant of the parallel Moser--Tardos Algorithm. We prove that if we restrict attention to a class of problems whose dependency graphs have subexponential growth, then the expected total number of random bits used by the algorithm is constant; in particular, it is independent of the number of variables. This is achieved by using the same random bits to resample variables which are far enough in the dependency graph. 
	
	There are two corollaries. First, we obtain a deterministic algorithm for finding a satisfying assignment, which for any class of problems as in the previous paragraph runs in time $O(n)$, where $n$ is the number of variables.  Second, we present a Borel version of the Lov\'asz Local Lemma.
\end{abstract}

		\maketitle

		\setcounter{tocdepth}{2}
		\tableofcontents
		
		\newcommand{\rnd}{\normalfont{\texttt{rnd}}}
		\newcommand{\RND}{\normalfont{\texttt{RND}}}

		\section{Introduction}
		
		In this article, we are interested in vertex colouring problems on graphs.  An instance of such a problem consists of a digraph $G$, a natural number $b\in\NN$ which is the number of colours which we use, and a system of constraints $\mathbf R$ which we call a \emph{local rule}, i.e.~for $x\in V(G)$ we have that $\mathbf R(x)$ is a set of $b$-valued functions defined on the out-neighbourhood of~$x$, where we also use $b$ to denote the set $\{0,\dots,b-1\}$.  We say that $f\in b^{V(G)}$  \textit{satisfies $\mathbf R$} if  for every $x\in V(G)$ the restriction of $f$ to the out-neighbourhood of $x$ belongs to~$\mathbf R(x)$. 
		
		One of the most useful tools for deducing that a colouring that satisfies a given local rule exists is the 
		\emph{Lov{\'a}sz Local Lemma} (LLL for short) first proved in~\cite{MR0382050}. Let us state the following version, which follows from~\cite{MR0491337}.
		
		Let $G$ be a digraph and let $\Rel(G)$ be the symmetric digraph whose vertex set is $V(G)$ and such that there is an edge between $x$ and $y$ if there is $z$ such that $(x,z),(y,z)\in E(G)$ (we allow $x$ and $y$ to be equal, so there may be self-loops in $\Rel(G)$).
		
		\begin{theorem}[Lov{\'a}sz Local Lemma~\cite{MR0491337}]\label{blll} 
			Let $G$ be a digraph and let $\De$ be the maximal vertex degree in $\Rel(G)$. If for every $x\in V(G)$ we have 
			\begin{equation}\label{eq:blll}
				1 - \frac{|\mathbf R(x)|}{b^{|\Var(x)|}} < \frac{1}{\me\De},
			\end{equation} where $\Var(x)$ denotes the out-neighbourhood of $x$,  then there exists $f\in b^{ V(G)}$ which satisfies~$\mathbf R$.
		\end{theorem} 
		
		\begin{myrem}
			Usually, LLL is stated for a ``satisfying assignment of variables'' instead of a satisfying colouring. Let us explain this alternative point of view and why it is equivalent to the formulation with satisfying colourings. 
			Let $W$ be a set of variables and let $C$ be a set of logical clauses, each of which depends on some of the variables in~$W$. The problem asks whether there exists an assignment $a\in b^W$ of values to the variables which makes all of the clauses true. The translation from such a ``satisfying assignment'' problem to a digraph colouring problem is by considering the bipartite digraph $G$ with $V(G):= C \sqcup W$, where $(x,y)\in C\times W$ is an edge if and only if  $y$ is a variable which appears in the clause $x$. For $x\in W$ the set $\Var(x)$ is empty and we set $\mathbf R(x):=b^\emptyset$, so that $\mathbf R(x)$ does not impose any restrictions on satisfying colourings. For $x\in C$ we let $\mathbf R(x)$ be the set of those assignments on $\Var(x)$ which make the clause $x$ satisfied. In the graph $\Rel(G)$ two clauses are adjacent if and only if they share a common variable, while all elements of $W$ are isolated vertices of $\Rel(G)$.
		\end{myrem}

		One of the key developments related to LLL is the Moser--Tardos Algorithm (MTA for short) studied by Moser and Tardos~\cite{MR2606086} (with a different version analysed earlier by  Moser~\cite{MR2780080}). The MTA is a randomised algorithm for \textit{finding} a satisfying colouring under the assumption that~\eqref{eq:blll} holds. Moser and Tardos~\cite{MR2606086} proved that if we restrict attention to a class of colouring problems where the difference between the sides in the inequality~\eqref{eq:blll} is at least a fixed constant $c>0$, then the MTA finds a satisfying assignment on a graph $G$ after using $O(|V(G)|)$ random bits in expectation.

		Moser and Tardos~\cite[Theorem 1.4]{MR2606086} also provided a deterministic polynomial time algorithm for instances of LLL when  the maximum degree of a dependency graph is uniformly bounded. This restriction was removed by Chandrasekaran, Goyal and Haeupler~\cite{ChandrasekaranGoyalHaeupler10,ChandrasekaranGoyalHaeupler13}. A running time is estimated in \cite[Theorem~5]{ChandrasekaranGoyalHaeupler13} to be, roughly speaking, $O(|V(G)|^{3+1/\eps})$ for some $\eps>0$ depending on the slacks in~\eqref{eq:blll}.

		\paragraph{Main result.} It is convenient to fix, for the rest of the article, a natural number $b>1$ which is the number of colours in our colouring problems. With this in mind, a \emph{colouring problem} is a pair $(G,\mathbf R)$, where $G$ is a digraph and $\mathbf R$ is a local rule on $G$. 
		
		The main aim of this article is to present a parallel version of the MTA and to show, in Theorem~\ref{thm-main} below, that it has the following property. Let $f\colon \NN\to \NN$ be a \emph{subexponential function}, i.e.~such that for every $\eps>0$ it holds that $\lim_{n\to\infty} \frac{f(n)}{(1+\eps)^n} =0$. Consider the class $\cal C=\cal C(f)$ of graphs in which any ball of a radius $r$ contains at most $f(r)$ vertices. For example, we could take $f(n):=(2n+1)^2$, in which case $\cal C$ would contain all subgraphs of the infinite $2$-dimensional grid. Let us also fix some $c>0$ and let $\cal P=\cal P(f,c)$ be the class of colouring problems $(G,\mathbf R)$ such that $G\in \cal C$ and 
		\begin{equation}\label{eq:gap}
			\max_{x\in V(G)} \left(1 - \frac{|\mathbf R(x)|}{b^{|\Var(x)|}} \right)\le \frac{1}{\me\De} - c.
		\end{equation} 
		Then there exists a constant $K>0$ such that for every colouring problem $(G,\mathbf R)\in \cal P$ the expected number of random bits which the algorithm uses is at most $K$ (where a \emph{bit} means an element of $b$). In particular, \textbf{the expected number of random bits used is independent from the number of vertices in $G$}. This is the crucial difference when compared to previous algorithmic versions of LLL. As a consequence, we obtain, in Corollary~\ref{cory-main}, a sequential \textbf{deterministic} algorithm which runs in time $O(|V(G)|)$ for colouring problems $(G,\mathbf R)$ in any class $\cal P$ of the kind we have just described.
		
		\paragraph{Connections to descriptive combinatorics and distributed algorithms.}
		Our results can also be used to derive a Borel version of the Lov\'asz Local Lemma. In brief, we say that a colouring problem $(G,\mathbf R)$ is \emph{Borel} if the vertex set $V(G)$ is a standard Borel space and both the edge set $E(G)\subset V(G)^2$ and the function $x\mapsto\mathbf R(x)$ are Borel. 
		In Theorem~\ref{thm-main-borel} we show that if $G$ has a uniformly subexponential growth and $(G,\mathbf R)$ is a Borel colouring problem for which the inequality~\eqref{eq:gap} holds, then there is a satisfying colouring $V(G)\to b$ which is a Borel function. This result appeared as the main result in the preprint~\cite[Theorem~4.7]{CsokaGrabowskiMathePikhurkoTyros:arxiv}, which is now superseded by this paper.

		Our Borel LLL is a result in 
		\emph{descriptive combinatorics}, an emerging field that studies, in particular, colourings of infinite graphs which are ``constructive'' in the sense of descriptive set theory. We refer the reader to an extensive survey of this field by Kechris and Marks~\cite{KechrisMarks20survey}. 
		
		Other descriptive versions of LLL were proved by Bernshteyn~\cite{Bernshteyn19am,Bernshteyn23a,Bernshteyn23i}, Bernshteyn and Weilacher~\cite{BernshteynWeilacher25} and Kun~\cite{Kun13arxiv}. For quick comparison with our Borel LLL, let us mention that the versions proved in \cite{Kun13arxiv,Bernshteyn19am,Bernshteyn23i} allow a null-set of errors, and the versions in~\cite{Bernshteyn23a,BernshteynWeilacher25}, while producing an error-free satisfying colouring which is respectively continuous and Borel, require a stronger condition than~\eqref{eq:gap}.\footnote{After the submission of this paper, Bernshteyn and Yu~\cite{BernshteynYu26} established (via a different method) a version of Borel LLL which is stronger than ours.}

		On the other hand, our version of LLL requires that the underlying graph has subexponential growth.  This might seem like a very strong assumption, especially since the other descriptive versions of LLL mentioned in the previous paragraph do not need any restriction of this kind. However, the assumption of subexponential growth cannot be removed in the Borel context in full generality (see Remark~\ref{rem:noborel}). Also, there is a considerable number of recent articles where graphs of subexponential growth are studied from the point of view of descriptive combinatorics, see e.g.~\cite{ConleyTamuz20,GrebikRozhon21a,GrebikRozhon23a,Thornton22}.

		Our Borel LLL from \cite{CsokaGrabowskiMathePikhurkoTyros:arxiv} 
		has already found some interesting applications. For example, it was crucially used by Bernshteyn and Yu~\cite{BernshteynYu25} to resolve some open problems on large-scale geometry of Borel graphs.  
		Also, Bernshteyn~\cite[Theorem 2.15]{Bernshteyn23i} used our lemma to show that, for graphs of subexponential growth, if a colouring problem on a graph $G$ can be solved by a randomised {\sf LOCAL}  algorithm (a model of distributed computing introduced by Linial~\cite{Linial92}) in $O(\log |V(G)|)$ rounds, then the corresponding Borel colouring problem admits a Borel solution.
		
		In fact, the Lov\'asz Local Lemma plays a very important role in the {\sf LOCAL} complexity landscape. 
		More specifically,  efficient distributed algorithms for solving LLL can be used for proving \emph{automatic speedup theorems} (i.e.~results that rule out whole intervals of possible complexities) for locally checkable labelling problems. In particular, if there exist distributed LLL algorithms matching the lower bound of Brandt et al~\cite{BFHKLRSU16} (as conjectured by Chang and Pettie~\cite[Conjecture~5.1]{ChangPettie19siamjc}), then a certain interval of complexity classes collapses to one point in the {\sf LOCAL} complexity landscape for general bounded-degree graphs. 
		We refer the reader to Chang and Pettie~\cite{ChangPettie19siamjc} for more detailed information.
		The best known upper complexity bounds for LLL (coming from the method of Fischer and Ghaffari~\cite{FischerGhaffari17} combined with the improved network decomposition algorithms by Ghaffari, Grunau and Rozho\v n~\cite{RozhonGhaffari20,GhaffariGrunauRozhon21}) are still polynomially far away from the conjectured values. The {\sf LOCAL} complexity of LLL is not fully understood also for the class of subexponential growth graphs. 
		It would be interesting to see if the ideas introduced in this paper could shed some light on this problem.

		\section{Algorithm and the results}\label{mtal}
		
		\subsection{Notation and conventions}
		We let  $\NN :=  \{0,1,2,\ldots\}$ and $\NN_+:=\NN\setminus \{0\}$. The cardinality of a set $X$ 
		is denoted by~$|X|$.   ``Either ... or ...'' is non-exclusive. The relation ``$\subset$'' allows equality.  We recall that if $n\in \NN$, then $n=\{0,1,\ldots,n-1\}$. 
		
		Given a set $X$, we denote by $\Pow(X)$  the power set of $X$, i.e.~the set of all subsets of~$X$. Given another set $Y$, we let $X^Y$ be the set of all functions from $Y$ to $X$.  In particular, we have the following notational clash: for $b,n\in \NN$, the symbol $b^n$ can either denote a number (and hence a set of numbers)  or the set of all functions from $\{0,1,\ldots, n-1\}$ to $\{0,1,\ldots, b-1\}$. We believe that resolving this ambiguity will never cause any difficulty for the reader. A convenient informal way of thinking about it is that if $k\in b^n$ then $k$ is both a number smaller than $b^n$ and a function from $n=\{0,\ldots, n-1\}$ to $b=\{0,\ldots, b-1\}$, and the translation between the two interpretations is that the $b$-ary expansion of the number $k$ can be thought of as a function from $\{0,\ldots, n-1\}$ to $\{0,\ldots, b-1\}$.
		
		Given a function $f$ with domain $D$ and a set $C\subset D$, we denote by $f \restriction C$ the restriction of $f$ to $C$.
		
		By a \emph{digraph} we mean a pair $G=(V,E)$ where $V$ is a non-empty set, and $E\subset V\times V$. In other words, our digraphs are without multiple edges, and each vertex is allowed to have at most one self-loop. A digraph $G=(V,E)$ is \emph{symmetric} if $E$ is a symmetric subset of $V\times V$. A reader who is not interested in the applications to Borel combinatorics in Section~\ref{sec-borel} may safely assume that $V$ is a finite set. In any case, each digraph that we consider has a finite maximal vertex degree. 
		
		A set $I\subset V(G)$ is \emph{independent} in $G$ if there are no edges in $G$, other than self-loops, which connect elements of $I$. 
		
		The \emph{graph distance} between vertices $x$ and $y$ is the minimal number of edges in a (not necessarily directed) path that connects $x$ and $y$. For $x\in V(G)$ and $R\in \NN$ we let  $N(x,R)$ be the \emph{ball of radius $R$ around $x$}, i.e.~the set of those $y\in V(G)$ whose graph distance to $x$ is at most $R$. 
		
		We use some non-standard notation for neighbourhoods of vertices. The motivation for it will become clear in the next subsection when we discuss local rules. For $x\in V(G)$ we define 
		\part  $\Var(x) := \{y\in V(G)\colon (x,y)\in E(G)\}$,  
		\part $\Cl(x) :=\{y\in V(G)\colon (y,x)\in E(G)\}$,
		\part $\deg(x) :=|\Var(x) \cup \Cl(x)|$, and 
		\part $\maxdeg(G) := \max_{x\in V(G)} \deg(x)$, \ \ $\maxoutdeg(G) :=\max_{x\in V(G)} |\Var(x)|$. 
		\trap
		
		If we need  to point out the dependence on $G$ then we use the notation $\Var_G(x)$, $\Cl_G(x)$ and $\deg_G(x)$. Let us stress again that we allow self-loops; in particular, $x$ may or may not be a member of either $\Var(x)$ or $\Cl(x)$.
		
		If $A$ is a set of vertices  then we let $\Var(A)  := \bigcup_{x\in A} \Var(x)$ and  $\Cl(A) := 
		\bigcup_{x\in A} \Cl(x)$.

		We define $\Rel(G)$ to be the digraph on the same set of vertices as $G$, and such that $(x,y)\in E(\Rel(G))$ if and only if $\Var_G(x)\cap \Var_G(y)\neq \emptyset$. Clearly, $\Rel(G)$ is a symmetric digraph and there is a self-loop at every $x$ such that $\Var_G(x) \neq \emptyset$.

		\subsection{Description of the algorithm}\label{mtalik}
		Our variant of the MTA depends on several choices. Some of them have very little influence on our analysis of the algorithm, but in this subsection, we take some time to spell out all the ingredients explicitly.
		
		\part We recall that we have fixed a natural number $b>1$, which is the number of colours in our colouring problems. A \emph{local rule} for a digraph $G$ is a function $\mathbf R$ whose domain is $V(G)$ and such that for every $x\in V(G)$ it holds that (i) $\mathbf R(x)\subset \bits^{\Var(x)}$, and (ii) if $\Var(x)=\emptyset$ then $|\mathbf R(x)|=1$, or equivalently $\mathbf R(x) = b^{\Var(x)}$. A \emph{colouring problem} is a pair $(G,\mathbf R)$, where $G$ is a digraph and $\mathbf R$ is a local rule on $G$. 
		
		Given a function $f\in b^{V(G)}$ and $x\in V(G)$, we let $\res_x (f) \in b^{\Var(x)}$ be the restriction of $f$ to $\Var(x)$. We say that $f$ \textit{satisfies} $\mathbf R$ if for every $x\in V(G)$ we have  $\res_x(f)\in \mathbf R(x)$.
		
		\begin{myrem}\begin{enumerate}[wide]
				\part Let us explain our choice of the notation $\Var(x)$ and $\Cl(x)$. First, a vertex $x\in V(G)$  has a colouring constraint attached to it, namely the set $\mathbf R(x)$. As such, we think of $\mathbf R(x)$ as a logical clause whose variables are the vertices in $\Var(x)$.  Second, it may be that $x\in \Var(y)$ for some $y\in V(G)$. In this case $\mathbf R(y)$ is a clause in which $x$ ``appears as a variable'', and this is the reason for the notation $\Cl(x)$: it is ``the set of all clauses which contain $x$ as a variable''.
				
				\part Condition (ii) in the definition of a local rule means that if $x$ is a clause without variables then $\mathbf R(x)$ does not impose any constraints on colourings. While Condition~(ii) can be omitted (namely, if it fails then Assumption~\eqref{eq-good} of Theorem~\ref{thm-main} does not hold and the theorem simply does not apply), it is convenient to have it.
			\end{enumerate}
		\end{myrem}
		Given $\mathbf R$, we let  $\mathbf R^c(x) :=\bits^{\Var(x)} \setminus \mathbf R(x)$. For $f\colon V(G)\to \bits$ we set 
		$$
		\B_{\mathbf R}(f):= \{x\in V(G)\colon \res_x(f) \in \mathbf R^c(x)\}.
		$$  
		We note that if $\Var(x)=\emptyset$ then for all $f$ we have $x\notin \B_{\mathbf R}(f)$. When $\mathbf R$ is clear from the context we write $\B(f)$ instead of $\B_{\mathbf R}(f)$. The notation $\B$ comes from ``bad set'', as $\B(f)$ is the ``bad set of $f$'', i.e.~the set of those clauses where $f$ does not fulfil $\mathbf R$.

		\part  An \textit{independence function} on a digraph $G =(V,E)$ is a function $\I=\I_G\colon \Pow(V) \to \Pow(V)$ with the property that for every $X\in \Pow(V)$ the set $\I(X)$ is a maximal subset of $X$ which is independent in $G$.
		
		To avoid clutter, we will write $\IB(f)$ instead of $\I_{\Rel(G)}(\B_{\mathbf R}(f))$ throughout the article. Thus $\IB(f)$ is a maximal collection of clauses violated by $f$ such that no two share a variable.
		
		\part A \textit{partition} of $V(G)$ is a finite set $\pi$ of pairwise disjoint subsets of $V(G)$ such that $V(G) = \bigcup_{W\in \pi} W$. We let $S_\pi\colon V(G) \to \pi$ be the  function defined by demanding that $x\in S_\pi(x)$ for all $x\in V(G)$. We say that a partition $\pi$  is \textit{$r$-sparse} if, for every $x\in V(G)$, different points of $N_G(x,r)$ belong to different elements of $\pi$.  Equivalently, $\pi$ is $r$-sparse if whenever $x\neq y$ and $x$ and $y$ are in the same part then the graph distance between $x$ and $y$ is greater than $2r$.
		\trap

		A \emph{Moser--Tardos tuple} is a tuple $(G,\mathbf R, \pi)$, where $G$ is a digraph, $\mathbf R$ is a local rule on $G$ and $\pi$ is a partition of $V(G)$. Our version of the MTA depends on a Moser--Tardos tuple $\cal M = (G, \mathbf R, \pi)$, an element $\rnd\in \bits^{\pi\times \NN}$ which can be thought of as the source of random bits, and an independence function on $\Rel(G)$. 
		
		\begin{myrem} Although the algorithm depends on the choice of the independence function $\I$ on $\Rel(G)$, we do not want to incorporate $\I$ in the notation, because the details of how independent sets are constructed on a particular graph are irrelevant for the analysis of the algorithm. As such we will just tacitly assume that we have a function $\I$ for every $(\cal M,\rnd)$; in fact it does not have to be a single function for the duration of the whole algorithm, it might also depend on the time, which will be useful in the proof of Corollary~\ref{cory-main} below.
		\end{myrem}
		
		The outcome of the MTA is a sequence of functions $\MT^i = \MT^i_{\cal M, \rnd}\in b^{V(G)}$ defined as follows.
		We let  $\MT^0(x) := \rnd(S_\pi(x),0)$ for all $x\in V(G)$, and we proceed to define the functions $\MT^i$ inductively. 
		
		Informally, if $\MT^i$ is already defined, then we take a maximal independent set of clauses where $\MT^i$ does not satisfy $\mathbf R$ (i.e.~a maximal subset of $\B_{\mathbf R}(\MT^i)$ which is independent in $\Rel(G)$), and we use $\rnd$ to assign new values to the variables which are in those clauses. In other words, we resample the values for variables in $\Var(\IB(\MT^i))$.
		
		Let us write the formal definitions now. We will also need a sequence of auxiliary functions $h^j= h^j_{\cal M,\rnd}\colon V(G) \to \NN$ counting the number of resamplings which the algorithm made at a given variable when defining $\MT^0,\MT^1,\ldots,\MT^{j-1}$.  We let $h^0(x):=0$ and $h^1(x):=1$ for all $x\in V(G)$. 
		
		Now suppose that for some $j \in \NN$ the functions $\MT^j$ and $h^{j+1}$ are defined. For $x\in V(G)$ we let
		\begin{equation}\label{eq-def-algorithm}
			\MT^{j+1}(x) := 	\left\{\begin{array}{l l}
				\rnd(S_\pi(x) ,h^{j+1}(x))& \quad \mbox{if $x\in \Var(\IB(\MT^j))$,}\\
				\MT^{j}(x)& \quad \mbox{otherwise,}\\ 
			\end{array} \right. 
		\end{equation}
		and 
		$$
		h^{j+2}(x) := 	\left\{\begin{array}{l l}
			h^{j+1}(x) +1& \quad \mbox{if $x\in \Var(\IB(\MT^{j}))$,}\\
			h^{j+1}(x) & \quad \mbox{otherwise. }\\ 
		\end{array} \right.
		$$
		
		This finishes the description of the algorithm. Note that for every $j\in \NN$ and $x\in V(G)$ it holds that $\MT^j(x)=\rnd(S_\pi(x),h^{j+1}(x)-1)$.
		
		Let us define $h^\infty =h^\infty_{\cal M,\rnd}\colon V(G) \to \NN\cup\{\infty\}$ by setting  $h^\infty(x) := \lim_{k\to\infty} h^k(x)$.  For  $\rnd\in b^{\pi\times \NN}$, we let $\rnd_l\in b^{\pi\times l}$ be the restriction of $\rnd$ to the set $\pi\times l = \pi \times \{0,\ldots, l-1\}$. 
		
		\begin{myrem}
			\part[wide,label=(\alph*)] The algorithm described by~\eqref{eq-def-algorithm} takes as its input a pair $(\cal M, \rnd)$, where $\cal M=(G,\mathbf R, \pi)$ is a Moser--Tardos tuple and $\rnd\in b^{\pi\times \NN}$.  Frequently it will be convenient to think of  $\cal M$ as being fixed and regard the algorithm as depending only on $\rnd$. 
			
			\part As described, the algorithm never terminates. This is because it is convenient in the proofs to have the functions $\MT^i$ defined for all $i\in\NN$.  However, by convention we say that our algorithm \emph{succeeds} after $i$ steps if $\B(\MT^{i-1})=\emptyset$. Note that if this is the case then $\MT^j=\MT^{i-1}$ for all $j\ge i-1$, and $\MT^{i-1}$ satisfies $\mathbf R$. 
			
			It will sometimes be convenient to only consider the first $k$ steps of the algorithm, i.e.~the algorithm described by~\eqref{eq-def-algorithm} terminated after producing $\MT^{k-1}$. This $k$-step MTA depends only on $\rnd_k\in b^{\pi\times k}$. 
			
			\part We note that $h^i(x)$ is defined so that the initial setting $\MT^0(x)=\rnd(S_\pi(x),0)$ is counted as a resampling of every variable. By definition, $h^\infty(x)$ is the total number of resamplings that the algorithm makes at the variable $x$.
			
			\part If for some $\rnd \in b^{\pi \times \NN}$ we have $l:= \sup_{x\in V(G)} h^{\infty}_{\cal M,\rnd}(x) <\infty$, then the algorithm succeeds for this instance of $\rnd$, and furthermore the algorithm depends only on $\rnd_l\in b^{\pi \times l}$.  
			
			\trap
		\end{myrem}

		\subsection{Statement of the main result}
		For $\nn,d\in \NN_+$ and $\eps>0$ let $\subexp(\nn,\eps,d)$ be the class 
		of all digraphs $G$ with $\maxdeg(G)\le d$, and such that for all $ x\in V(G)$ we have 
		$$
		|N_G(x,3\nn)|\le (1+\eps)^{\nn}.
		$$
		
		Our main result is the following theorem, where (as everywhere else in the paper) we define $\De := \maxdeg(\Rel(G))$, with this also taking into account the loops of $\Rel(G)$, and assume that $\De>0$ as otherwise there are no restrictions on a satisfying assignment. Also, let $\Pr_{\rnd}$ be the probability measure where we take uniform random $\rnd\in b^{\pi\times\NN}$.
		
		\begin{theorem}\label{thm-main}
			Let $\cal M=(G,\mathbf R, \pi)$ be a Moser--Tardos tuple, let $\De := \maxdeg(\Rel(G))>0$, and let integers $\nn,d\in \NN_+$ and  reals $\eps,\delta>0$  be such that 
			\part $G\in \subexp(\nn,\eps,d)$,
			\part the partition $\pi$ is $3\nn$-sparse, and
			\part it holds that
			\beq[eq-good]
			\be_{\cal M} \le \frac{1}{(\me\De)^{1+\delta}\, b^{\eps d}},
			\eeq
			where we define
			\begin{equation}\label{eq:beta}
				\be_{\cal M} = \be_{(G,\mathbf R)} := \max_{x\in V(G)} \frac{|\mathbf R^c(x)|}{|b^{\Var(x)}|}.
			\end{equation}
			\trap
			Then there is a constant $K$ which depends only on  $b$, $\de$, $d$ and $|\pi|$, such that for all $ m\in \NN_+$ we have that 
			\begin{equation}\label{eq-wqeq1}
				\Pr_{\rnd}\left(\sup_{x\in V(G)} h^\infty_{\cal M,\rnd}(x)> m\right) \le \frac{K(m+1)^{|\pi|}} {(\me\De)^{\de m}}.
			\end{equation}%
		\end{theorem}
		
		The proof is presented in Section~\ref{sec-analysis}. 
		
		\begin{myrem}\label{rem-after-main}
			
			\part Since $|\Var(x)|\le d$ is uniformly bounded for $x\in V(G)$, we can take maximum in the left-hand side of~\eqref{eq:beta} (rather than supremum).

			\part Let us comment on the cardinality $|\pi|$, i.e.~the number of parts in the partition~$\pi$. Let $H$ be the graph on the same set of vertices as $G$, where we connect two distinct vertices if they are at a distance at most $6\nn$ (with respect to the graph distance in~$G$). Thus $H$ has maximal vertex degree bounded by $1+d\,\sum_{i=1}^{6\nn}(d-1)^{i-1}$ which is at most, say,~$d^{6\nn}+1$. Clearly, a partition is $3\nn$-sparse if and only if it induces a proper vertex colouring of~$H$. Thus a simple greedy colouring procedure shows that there is a $3\nn$-sparse partition $\pi$ with $|\pi|\le d^{6\nn}+2$.
			
			\part\label{pa:c}  If $G$ is a finite digraph then the bound~\eqref{eq-wqeq1} implies that if $m\in \NN_+$ is such that 
			\begin{equation}\label{eq:m}
				\frac{K(m+1)^{|\pi|}} {(\me\De)^{\de m}}< 1,
			\end{equation}
			then there exists $\rnd$ such that the algorithm's run on $(\cal M,\rnd)$ finds a satisfying colouring and accesses only the elements of~$\rnd_m$. 
			
			\trap 
		\end{myrem}

		Remark~\ref{rem-after-main}\,\ref{pa:c} leads to a sequential deterministic MTA which 
		runs in time $O(|V(G)|)$ for many interesting classes of vertex colouring problems.
		To be precise, let us define the following class of colouring problems (we recall that we have fixed once and for all the number of colours $b$ for the local rules). For $\nn,d\in \NN_+$ and $\de >0$, let $\colsubexp(\nn,d,\de)$ be the class of 
		colouring problems $(G,\mathbf R)$, where $G$ is a finite digraph, and for which there is $\eps>0$ such that
 $G\in \subexp(\nn,\eps,d)$ and
		$$
		\be_{(G,\mathbf R)}\le \frac{1}{(\me\De)^{1+\delta}\, b^{\eps d}}.
		$$
		
		\trap
		\begin{cory}\label{cory-main}
			For every $(\nn,d,\de)$ there is a sequential deterministic algorithm that takes as an input a colouring problem $(G,\mathbf R)\in \colsubexp(\nn,d,\de)$, and finds a satisfying assignment in time $O(|V(G)|)$ (where the time complexity is with respect to the standard random-access machine computational model, see e.g.~\cite[Section 2.6]{books/daglib/0072413}).
		\end{cory}

		\begin{proof}
			Let us describe the algorithm. Let $(G,\mathbf R)$ be its input. 
			We start by finding a $3\nn$-sparse partition of $V(G)$ with $|\pi|\le d^{6\nn}+2$. This can be done in linear time by first computing the undirected graph power $G^{6\nn}$ and applying the greedy algorithm for finding a vertex colouring in $G^{6\nn}$. Let $m$ be a sufficiently large constant to satisfy $K(m+1)^{|\pi|}< \me^{\de m}$. Thus $m$ depends only on $b,\nn, d$ and $\de$, and satisfies~\eqref{eq:m} since $\De\ge 1$. 
			
			Let $\cal M: = (G, \mathbf R, \pi)$ and let $\eps>0$ witness $(G, \mathbf R)\in \colsubexp(\nn,d,\de)$. By Theorem~\ref{thm-main},  there exists $\rnd\in  b^{\pi\times \NN}$ such that the algorithm's run on $(\cal M, \rnd)$ succeeds and accesses only $\rnd_m\in b^{\pi\times m}$. 
			
			We have the following deterministic algorithm. The first ``external loop'' runs over all (constantly many) choices of $\rnd_m\in b^{\pi\times m}$. For each such $\rnd$, the ``internal loop'' runs our algorithm on $(\cal M,\rnd_m)$ so that each pass of the internal loop does a single step of the MTA algorithm. We run the internal loop
			until one of the two things happens: either  $\B(\MT^i)$ is empty at some moment, in which case we halt the whole algorithm and output $\MT^i$, or the algorithm needs to access $\rnd( \ast, m)$, that is, the $(m+1)$-st bit of some vertex, in which case we move to the next $\rnd_m\in b^{\pi\times m}$. Note the internal loop resamples at most $|V(G)|\cdot m$ variables and thus stops after at most $|V(G)|\cdot m$ passes.  
			
			Let us describe in more detail how  a single pass of the internal loop can be implemented.    
			
			Before the internal loop starts we define two lists: Currently-Violated and Potentially-Violated.  Currently-Violated is the list of all clauses which are violated by $\MT^0$; to define Potentially-Violated we start with Currently-Violated and we append all neighbours in $\Rel(G)$ of the elements in Currently-Violated. The role of Potentially-Violated is that if a clause will be violated by $\MT^1$ then it is necessarily in Potentially-Violated.  The lists Currently-Violated and Potentially-Violated will be updated at the end of each pass of the internal loop. 
			
			For each pass of the internal loop, we go through the clauses in Currently-Violated, and for each clause $C$ we check if we have already encountered, in the current pass, a resampled clause $D$ which is a neighbour of $C$ in $\Rel(G)$. If this is the case then we move to the next clause in Currently-Violated, and otherwise, we resample the variables of $C$ using $\rnd$. 
			
			At the end of the pass, we define new lists Currently-Violated and Potentially-Violated as follows: first, we let new Currently-Violated consist of the entries of old Potentially-Violated which are violated, and then we let new Potentially-Violated consist of the new list Currently-Violated and all the neighbours in $\Rel(G)$ of the elements of the new list Currently-Violated. 
			
			We terminate the internal loop if 1) at any point $\rnd(\ast, m)$ needs to be accessed, and then we move to the next $\rnd$ in the external loop, or 2) after some pass the list Currently-Violated is empty (which means that we found a satisfying assignment).
			
			This finishes the description of the sequential deterministic algorithm. Let us argue that the internal loop indeed implements a single step of the parallel deterministic algorithm, that is, we resample variables corresponding to a maximal independent set of violated clauses. Since we resample each inspected clause if and only if it is violated and has no previously resampled neighbours in $\Rel(G)$, it remains to argue that every violated clause was inspected. 
			For this we note that for every $i\in \NN$ we have that before the $i$-th pass of the internal loop commences, the list Potentially-Violated contains all the clauses that are violated, together with all clauses that might be violated after the $i$-th pass finishes. Thus before the $i$-th pass commences, the list Currently-Violated contains all the violated clauses.

			Let us show that
			the presented sequential deterministic algorithm runs in linear time  in~$|V(G)|$. Of course, the initial computation of the lists Currently-Violated and Poten\-tially-Violated from $\MT^0$ takes linear time. Also, in order to update these two lists during the $i$-th pass (when we compute and resample a maximal independent subset of $\MT^{i-1}$) for $i=1,2,\ldots$\,, the algorithm needs to check only those clauses that share a variable with at least one clause in  $\B(\MT^{i-1})$, which is exactly the list Potentially-Violated at the beginning of the pass.
			Thus the total number of clause re-evaluations in our sequential algorithm (after $\MT^0$ and $\B(\MT^0)$ have been computed) is at most 
			$$
			d^2\sum_{i=1}^\infty |\B(\MT^{i-1})|\le d^4\sum_{i=1}^\infty |\IB(\MT^{i-1})|\le d^4m\,|V(G)|,
			$$ 
			where the second inequality uses the fact that the internal loop resamples at most $m\,|V(G)|$ variables. Thus the sequential algorithm indeed runs in linear time.\end{proof}

		\section{Analysis of the algorithm}\label{sec-analysis}
		
		Let $\cal M= (G,\mathbf R, \pi)$ be a Moser--Tardos tuple. 
		
		\begin{myrem} We will first give the proof of Theorem~\ref{thm-main} when $G$ is a finite digraph and $\pi$ consists of singletons of all vertices of $G$. This will be done in Subsection~\ref{subsec-interludium}.  When $\pi$ consists of singletons of all vertices of $G$ then our variant of MTA described in Subsection~\ref{mtalik} is essentially the same as the original Moser--Tardos algorithm.
			
			As such, the analysis of the algorithm is very similar to the analysis of other versions of the MTA that exist in the literature: we associate to each $(\cal M,\rnd)$  a combinatorial gadget $\cal L_{\cal M,\rnd}$ which we call a landscape, which essentially consists of a forest with decorations. The most important feature of $\cal L_{\cal M, \rnd}$ is that we can recover the part of $\rnd$ which was utilised in the run of the algorithm on $(\cal M, \rnd)$. This, together with bounds on the number of landscapes of a given size, will lead to the desired bound.
			
			The main reason to prove Theorem~\ref{thm-main} first under the extra assumptions is to explain what needs to be modified for the general case of Theorem~\ref{thm-main}. This explanation is presented in Subsection~\ref{se:ideas} and we hope that it motivates well the remaining constructions which involve landscapes, which are presented in Subsections~\ref{subsec-restrictions} and~\ref{subsec-equiv}. The proof of the general case of Theorem~\ref{thm-main} will be presented in Subsection~\ref{subsec-mainproof}.
		\end{myrem}
		
		Let $b^{\oplus \NN}$ be the set of all finite sequences consisting of elements of $\{0,\dots,b-1\}$, where we index each sequence with an initial segment of $\NN$. If $s=(s_0,\ldots, 
		s_{k-1})\in b^{\oplus \NN}$ then its \emph{length} is $\len(s) :=k$. For a set $X$ and  
		$f\colon X \to b^{\oplus \NN}$ we let $\len(f):= \sum_{x\in X} 
		\len(f(x))\in\NN\cup\{\infty\}$. More generally, if $Y\subset X$ then we let $\len_Y(f) := \sum_{y\in Y} \len(f(y))$.   
		We say that a function $g\colon X\to b^{\oplus \NN}$  \emph{$k$-complements} $f$ on $Y\subset X$ if for every $x\in Y$ we have  that $\len(f(x))+\len(g(x))  =k$. 
		
		\begin{myrem}\label{rem-lengths}
			We will use  the following simple observation several times: if we fix non-negative numbers $\len(f(x))$ for $x\in X$ with $\len(f)$ finite, then there are exactly $b^{\len(f)}$ possibilities for $f$. 
		\end{myrem}
		
		Let $k\in \NN_+$ and let $\rnd\in b^{\pi\times k}$. We will now define 
		two functions, $\used^k_{\cal M,\rnd}$  and $\unused^k_{\cal M, \rnd}$ 
		which keep track of the bits of $\rnd_k$  used when defining 
		$\MT^0,\ldots, \MT^{k-1}$ and the remaining bits respectively. Both functions are defined on $V(G)$ and have 
		values in~$b^{\oplus \NN}$. For $x\in V(G)$ we define 
		$\used^k_{\cal M,\rnd}(x)$ to be the sequence 
		$$
		\rnd(S_\pi(x),0),\ldots, \rnd(S_\pi(x),h^{k}(x)-1).
		$$
		Also, we define $\unused^k_{\cal M,\rnd}(x)$ to be the sequence 
		$$
		\rnd(S_\pi(x), h^{k}(x)),\ldots, \rnd(S_\pi(x),k-1).
		$$
		In particular, if $h^{k}(x)=k$, the maximum possible value, then $\unused^k_{\cal M,\rnd}(x)$ is the empty sequence. 
		Note that $\rnd_k(S_\pi(x),\ast)$ is the concatenation of $\used^k_{\cal M,\rnd}(x)$ and $\unused^k_{\cal M,\rnd}(x)$.
		
		\subsection{Landscapes}\mbox{}

		\begin{myrem} A landscape is a combinatorial gadget which arises from $(\cal M,\rnd)$,   roughly  speaking, as follows. 
			Recall that the algorithm at every step takes a maximal independent set $S$ of 
			clauses that are violated, and resamples all variables which appear in 
			these clauses. After this, every invalid clause shares variables with some clause in~$S$.
			
			This leads to a structure of a directed graph: vertices are the pairs $(C,i)$, where $i$ is a step, and 
			$C$ is a clause that is invalid after step $i$; and edges 
			$(C,i)\to (D,i+1)$ correspond to the ``causal relation'' which could 
			be very informally thought of as ``$C$ was invalid after step $i$; we resampled $C$ in step $i+1$ and this made $D$ invalid after step $i+1$''.
			
			We trim this directed graph to a forest and then add some decorations: (a) each variable remembers its final assignment, (b) each vertex $(C,i)$ remembers the values of the variables of $C$ in the colouring $\MT^i$. 
			
			The main point of associating a landscape to $(\cal M,\rnd)$ is that it is possible to recover the function $\used^k_{\cal M,\rnd}$ from the landscape (or, more precisely, from the finalised landscape).
		\end{myrem}
		
		Let us now proceed with precise definitions. Let $G$ be a digraph. We define a digraph $\Canvas(G)$ as follows. The set of vertices of $\Canvas(G)$ is $V(G)\times \N$. For every edge $(x,y)$ in $\Rel (G)$  and every $i\in \NN$ we add an edge $((x,i), (y,i+1))$ to $\Canvas(G)$. Note that, since the digraph $\Rel (G)$ is symmetric, the edge $((y,i), (x,i+1))$ is also added  to $\Canvas(G)$.    
		If $(x,i)\in V(G)\times \NN$  then we refer to $i$ as the \emph{level} of $(x,i)$. 
		
		A \emph{$G$-forest} is a subdigraph $\cal F$ of $\Canvas(G)$ such that each vertex of $\cal F$ has in-degree equal to either $0$ or $1$.  We note that $\cal F$ is a forest, and we let $\Trees(\cal F)$ be the set of connected components of $\cal F$. For $\tau \in \Trees(\cal F)$ we let $\rho(\tau) \in V(\tau)$ be the \emph{root}, i.e.~the unique vertex with minimal level, and we define $\ell(\tau)\in \NN$ to be the level of $\rho(\tau)$.  If for all $\tau\in \Trees(\cal F)$ we have $\ell(\tau)=0$ then we say that $\cal F$ is \emph{grounded}.
		The \emph{height} of $\cal F$ is the minimal $j\in \NN\cup\{\infty\}$ such that for all $(x,i)\in V(\cal F)$  we have $i< j$. In particular,  the unique $G$-forest with no vertices has height $0$. 
		We say that $\cal F$ is \emph{independent} if for all $i\in \NN$ the set $\{x\in V(G)\colon (x,i)\in V(\cal F)\}$ is an independent set in $\Rel(G)$. 
		
		Let $\cal M=(G, \mathbf R, \pi)$ be a Moser--Tardos tuple. An \emph{$\cal M$-landscape} is a pair $\cal L =(\cal F,\prev)$, where $\cal F$ is an independent $G$-forest,  and $\prev$ is a function with domain $V(\cal F)$ such that  for all $(x,i) \in V(\cal F)$ we have $\prev(x,i)\in \mathbf R^c(x)$.  
		
		\begin{myrem} Let us informally motivate the choice of the name $\prev$. It is a shorthand for ``Violation''. We will shortly see that, in the landscape associated to a run of the algorithm, we have that  $\prev(x,i)$ encodes the way in which the clause $x$ is violated in~$\MT^i$.
		\end{myrem} 
		
		A \emph{finalised $\cal M$-landscape} is a tuple $(\cal F,\prev, \final)$ where $(\cal F, \prev)$ is an $\cal M$-landscape of finite height, and $\final\in b^{V(G)}$.
		
		We say that a landscape $\cal L=(\cal F,\prev)$ or a finalised landscape $\cal L=(\cal F,\prev,\final)$ is \emph{grounded}, or of \emph{finite height}, if $\cal F$ has the respective property. Furthermore we define $V(\cal L) := V(\cal F)$ and $\Trees(\cal L) :=\Trees(\cal F)$.

		\subsubsection{Landscapes associated to a run of the MTA}
		
		Let  $\cal M = (G, \mathbf R,  \pi)$ be a Moser--Tardos tuple, and let $\rnd\in b^{\pi\times \NN}$. Let $<$ be a total order on $V(G)$. We now proceed to define an  $\cal M$-landscape associated to the run of the algorithm on $(\cal M,\rnd, <)$. The order plays a very minor role in the definition of this landscape and the subsequent analysis, so we will be slightly imprecise and denote it by $\cal L_{\cal M,\rnd}$, i.e.~without incorporating the order into the notation.
		
		We start with defining an independent $G$-forest $\cal F = \cal F_{\cal M,\rnd}$. First, we let 
		$$
		V(\cal F): = \bigcup_{i\in \NN} \left(\IB(\MT^i)\times\{i\}\right).
		$$
		Let us now describe the edges in $\cal F$. First, we need the following lemma.
		\begin{lemma}\label{lem-13} If $(x,i+1)\in V(\cal F)$ with $i\ge 0$ then   
			there exists $(y,i)\in V(\cal F)$ such that $\Var_G(x)\cap \Var_G(y)\neq \emptyset$ (where we do not demand that $x\neq y$). 
		\end{lemma}
		\begin{proof}
			Let us  assume, by way of contradiction, that there is no $(y,i)\in V(\cal F)$ such that $\Var(x)\cap \Var(y) \neq \emptyset$. In particular,  since $\Var(x) \neq \emptyset$, we have that $x\notin \IB(\MT^i)$.
			
			It follows that no variables of the clause $x$ were resampled when passing from $\MT^i$ to~$\MT^{i+1}$. Hence, $x\in \B(\MT^i)$. But $\IB(\MT^i)\not\ni x$ is a maximal independent subset of $\B(\MT^i)\ni x$, so $\IB(\MT^i)$ must include some $y\in N_{\Rel(G)}(x,1)$, a contradiction.
		\end{proof}
		
		Now, for every $(x,i+1)\in V(\cal F)$ with $i\ge 0$ we take the minimal element $y$ of 
		the set $N_{\Rel(G)}(x,1)\cap \IB(\MT^i)$ (which is non-empty by the previous lemma and finite by $\Delta<\infty$), and add the edge $(y,i)\to 
		(x,i+1)$ to~$\cal F$. This finishes the definition of $\cal F = \cal 
		F_{\cal M,\rnd}$.
		
		Now  let $\prev_{\cal M,\rnd}(x,i)$ for $(x,i)\in V(\cal 
		F_{\cal M,\rnd})$ be the function $\MT^i$ restricted to $\Var(x)$. This finishes the construction of the $\cal M$-landscape $\cal L_{\cal M,\rnd} = (\cal F_{\cal M,\rnd}, \prev_{\cal M,\rnd})$. Lemma \ref{lem-13} implies that this landscape is grounded.
		
		We also define, for every $k\in\NN_+$, a finalised $\cal M$-landscape  
		\begin{equation}\label{eq-fin-landscape}
			\cal L^k_{\cal M,\rnd} :=(\cal F^{k-1}_{\cal M,\rnd},\prev_{\cal M, \rnd}\restriction V(\cal F^{k-1}_{\cal M,\rnd}), \MT^{k-1}),
		\end{equation}
		where $\cal F^j_{\cal M,\rnd}$ for $j\in\NN$ is the digraph induced by $\cal F_{\cal M,\rnd}$ on the set 
		$$
		\bigcup_{i < j} \left(\IB(\MT^i)\times\{i\}\right).
		$$
		In particular $\cal F^0_{\cal M,\rnd}$ is the empty digraph, and thus the only interesting data in $\cal L^1_{\cal M, \rnd}$ is the assignment $\MT^0$.

		\subsubsection{The sequence encoded by a finalised landscape}
		
		Let  $\cal M=(G, \mathbf R,\pi)$ be a Moser--Tardos tuple and let $\cal L =(\cal F,\prev, \final)$ be a finalised $\cal M$-landscape of finite height $k$. We  define the function $\used_{\cal L}\colon V(G)\to b^{\oplus \NN}$ as follows. For $x\in V(G)$, we consider two cases. 
		
		If there does not exist $(y,i)\in V(\cal F)$ with $x\in \Var(y)$ then we let $\used_{\cal L}(x) := (\final(x))$, i.e.~$\used_{\cal L}(x)$ is a sequence of length $1$.  
		
		If there is $ (y,i)\in V(\cal F)$ with $x\in \Var(y)$, we start by listing all vertices $(y_1,s_1)$, $(y_2,s_2)$,$\ldots$, $(y_t,s_t)\in V(\cal F)$, such that $x\in \Var(y_j)$. Since the forest $\cal F$ is independent, each natural number appears at most once as the second coordinate in this sequence, so we may assume that $s_1<s_2<\ldots<s_t$. Now we define $\used_{\cal L}(x)$ as the sequence 
		\[
		\used_{\cal L}(x):=(\prev(y_1,s_1)(x), \prev(y_2,s_2)(x),\ldots, \prev(y_t,s_t)(x), \final(x)).
		\] 
		
		The following lemma follows directly from the definitions of $\used_{\cal L_{\cal M,\rnd}^k}$ and $\used_{\cal M,\rnd}^k$. Its meaning is that the random bits that are used in the first $k$ steps of the algorithm can be recovered from the finalised landscape $\cal L_{\cal M,\rnd}^k$.
		
		\begin{lemma}\label{lem-equal-used}
			Let $\cal M =(G, \mathbf R,  \pi)$ be a Moser--Tardos tuple. For all $k\in\NN_+$ and all $\rnd\in b^{\pi\times \NN}$ we have 
			$$
			\used_{\cal L^k_{\cal M,\rnd}} = \used_{\cal M,\rnd}^k \, ,
			$$
			where $\cal L^k_{\cal M,\rnd}$ is defined by~\eqref{eq-fin-landscape}.\qed
		\end{lemma}
		
		For any $G$-forest $\cal F$, let us define 
		$$
		\varcount(\cal F) :=  |V(G)|+ \sum_{(y,i)\in V(\cal F)} |\Var(y)| \in \NN\cup \{\infty\},
		$$
		and let $\varcount(\cal L):=\varcount(\cal F)$ for a finalised landscape $\cal L=(\cal F,\prev,\final)$. The meaning of the quantity $\varcount(\cal L)$, when $\cal L=\cal L^k_{\cal M,\rnd}$, is that it is equal to the total number of variable assignments which the algorithm's run on $(\cal M,\rnd)$ resamples when defining $\MT^0,\ldots, \MT^{k-1}$. 
		
		\begin{lemma}\label{lem-varcount}
			Let $\cal M=(G,\mathbf R, \pi)$ be a Moser--Tardos tuple, and let $\cal L=(\cal F, \prev,\final)$ be a finalised $\cal M$-landscape. We have
			\begin{equation}\label{eq-239}
				\len(\used_{\cal L}) = \varcount(\cal F).
			\end{equation}
			Furthermore, if $\cal L=\cal L^k_{\cal M,\rnd}$ then
			\begin{equation}\label{eq-022}
				\sum_{x\in V(G)} h^k_{\cal M, \rnd} (x)= \len(\used_{\cal L}).
			\end{equation}
		\end{lemma}
		\begin{proof}
			The equality~\eqref{eq-239} represents the following double counting: the right-hand side  is
			$$
			|V(G)| + \sum_{(y,i)\in V(\cal F)} |\Var(y)|,
			$$
			whereas the left-hand side is equal to
			$$
			\sum_{x\in V(G)} \len(\used_{\cal L}(x))= \sum_{x\in V(G)} \big(1+ |\{(y,i)\in V(\cal F)\colon x\in \Var(y)\}|\big).
			$$
			The fact that we are not overcounting on the left-hand side follows from the fact that $\cal F$ is independent, i.e.~for fixed  $x\in V(G)$ and $i\in\NN$ there is at most one $(y,i)\in V(\cal F)$ such that $x\in \Var(y)$. This finishes the proof of~\eqref{eq-239}.
			
			We note that 
			$$
			h^k_{\cal M, \rnd} (x) = \len(\used^k_{\cal M,\rnd}(x)),
			$$
			and so the left-hand side of~\eqref{eq-022} is equal to $\len(\used^k_{\cal M,\rnd})$. Since Lemma~\ref{lem-equal-used} shows that  $\used_{\cal M,\rnd}^k = \used_{\cal L^k_{\cal M,\rnd}}$, the equality~\eqref{eq-022} follows.
		\end{proof}

		\subsection{Counting landscapes}
		
		Let $\cal M = (G, \mathbf R,  \pi)$ be a Moser--Tardos tuple. For a given $m\in \NN$, we want to bound the number of grounded finalised $\cal M$-landscapes $(\cal F,\prev, \final)$ such that $|V(\cal F)|= m$. The most important part is counting the possibilities for $\cal F$, and this is what we proceed to do now. 
		
		\newop{indeg}{indeg}
		
		Let $\De\in \NN_+$. A \textit{$\De$-labelled tree} is a finite directed tree such that at each vertex the in-degree is either $0$ or $1$, and such that the edges are labelled by the elements of $\De=\{0,\ldots, \De-1\}$, in such a way that at each vertex all the out-going edges have different labels. The \emph{root} of a $\De$-labelled tree is the unique vertex whose in-degree is equal to $0$.
		
		It is convenient to regard the empty digraph as a $\De$-labelled tree on $0$ vertices. An \emph{isomorphism} between two $\De$-labelled trees is a graph isomorphism that preserves directions and edge labels (and thus, in particular, it preserves the root). By the \emph{iso-class} of $T$ we will mean the class of all trees isomorphic to~$T$.
		
		\begin{lemma}\label{lem-count-dtrees}
			\part  For $i\in \NN$  let $P_i$ be the number of iso-classes of $\De$-labelled trees with $i$ vertices. We have $P_i \le (\me\De)^i$.
			\part For $k,i\in \NN$ let $R_{k,i}$ be the number of all finite sequences $(T_0,\ldots, T_{k-1})$ of iso-classes of $\De$-labelled trees such that $\sum_{j=0}^{k-1} |V(T_j)| = i$.  We have $R_{k,i}\le (i+1)^{k-1}(\me\De)^i$.
			\trap
		\end{lemma}
		
		\begin{proof} If $\De=1$ then $P_i=1$ and the lemma holds, so assume that $\De\ge 2$.
			
			\part[wide,label=(\alph*)] This proof is based on the presentation in~\cite{spencer-note}. In fact, we will show that  $P_i \le \left(\frac{\De^\De}{(\De-1)^{\De-1}}\right)^i$, which implies the claim by $(\frac{\De}{\De-1})^{\De-1}\le \me$ for every $\De\ge 1$.
			
			Let us consider the formal power series:
			$$
			P(X) := \sum_{i=0}^\infty P_i X^i.
			$$
			Let $\rho := \frac{(\De-1)^{\De-1}}{\De^\De}$. Since $P_i\ge 0$ and $P_0=1$, it is enough to show that $P$ converges at $\rho$ and $P(\rho) \le P_0+1=2$. In fact  we will show that $P(\rho)\le \frac{\De}{\De-1}$.

			For $i,j\in \N$ let us define $Q_{i,j}\in \N$ as the number of $\De$-labelled trees with $i$ vertices and such that all vertices are at the distance at most $j$ from the root. By convention we set $Q_{0,j}:=1$ for all $j\in \NN$.
			Let 
			$$
			Q_j(X) :=\sum_{i=0}^\infty Q_{i,j}X^i.
			$$
			We have that $Q_j$'s are polynomials since $Q_{i,j} = 0$ when $i$ is large enough. The following claim is clear from the definitions.
			\begin{claim*} 
				The polynomials $Q_j$ have non-negative real coefficients, and for all $n\in \NN$  and all $j>n$ we have that the first $n$ coefficients of $Q_j(X)$ are equal to the first $n$ coefficients of $P(X)$.\qed
			\end{claim*}
			
			Since both $Q_j$'s and $P$ have non-negative coefficients, we obtain the following claim as a corollary.
			
			\begin{claim*} 
				If for some $x>0$ the sequence $Q_0(x),Q_1(x),\ldots$ is bounded by $y\in \RR$, then the power series $P(X)$ converges at $x$ and we have $P(x) \le y$.\qed
			\end{claim*}
			
			We proceed to use the above claim to establish that  $P(\rho)\le \frac{\De}{\De-1}$. For this let us show by induction that for all $j\in \NN$ we have $Q_j(\rho)\le  \frac{\De}{\De-1}$.  Since $Q_0(X) = 1+ X$, we have  $Q_0(\rho)=1+ \rho \le 1+\frac{1}{\De-1} =\frac{\De}{\De-1}$. Thus let us assume that for some $j\in\NN$ we have $Q_j(\rho)\le \frac{\De}{\De-1}$. 
			
			Note that a $\Delta$-labelled tree whose all vertices are at a distance at most  $j+1$ from the root is either the empty tree, or consists of the root and a sequence of exactly $\De$ trees (some of which may be empty), each of which has all vertices at distance at most $j$ from its root. This leads us to the following equation valid for all $j\in \NN$:
			$$
			Q_{j+1}(X) =1+X\cdot Q_j(X)^\De.
			$$
			Thus, using the inductive assumption, we have 
			$$
			Q_{j+1}(\rho) \le 1+\rho\cdot \left(\frac{\De}{\De-1}\right)^\De = \frac{\De}{\De-1},
			$$
			which finishes the proof of  (a).
			
			\part 
			
			Let us set $R_k(X) := \sum_{i=0}^\infty  R_{k,i} X^i$. Then we have $  R_k(X) = P(X)^k$. Thus by the previous part, we can bound the coefficients of $ R_k$ from above by the respective coefficients of the series 
			$$
			S_k(X):=\left(\sum_{i=0}^\infty (\me\De X)^i\right)^k.
			$$
			Let us show by induction on $k$ that the $i$-th coefficient $S_{k,i}$ of $S_k(X)=\sum_{i=0}^\infty S_{k,i} X^i$ is bounded from above by $(i+1)^{k-1} (\me\De)^i$. The inductive statement is clearly true for $k=0$ and $1$. Let us assume that it is true for some $k\ge 1$. We write
			$$
			S_{k+1}(X) = S_k(X) \cdot \left(\sum_{i=0}^\infty (\me\De X)^i\right),
			$$
			and so 
			\begin{align*}
				S_{k+1,i} = \sum_{j=0}^i S_{k,j}\,  (\me\De)^{i-j} &\le \sum_{j=0}^i (j+1)^{k-1}(\me\De)^j (\me\De)^{i-j} 
				\\
				&\le (i+1) (i+1)^{k-1} (\me\De)^i 
				\\
				&= (i+1)^k (\me\De)^i,
			\end{align*}
			which finishes the proof of the inductive statement, and the proof of (b).\qedhere
			\trap
		\end{proof}
		
		\begin{lemma}\label{lem-count-landscapes}
			Let $G$ be a finite digraph, let $\cal M= (G,\mathbf R, \pi)$ be a Moser--Tardos tuple, let $\be=\be_{\cal M}$ as defined in~\eqref{eq:beta} and let $\cal F$ be a $G$-forest of finite height. The number of $\cal M$-landscapes $(\cal F, \prev)$ is at most 
			$
			\be^{|V(\cal F)|}\cdot b^{\varcount(\cal  F) -|V(G)|}.
			$
		\end{lemma}
		\begin{proof} For every $(x,i)\in V(\cal F)$ we have that $\prev(x,i)\in \mathbf R^c(x)$. Thus there are at most
			$$
			\prod_{(x,i)\in V(\cal F)} |\mathbf R^c(x)| = \prod_{(x,i)\in V(\cal F)} \left(\frac{|\mathbf R^c(x)|}{|b^{\Var(x)}|}\cdot \left|b^{\Var(x)}\right| \right) 
			$$
			possibilities for $\prev$, which is bounded from above by
			$$
			\be^{|V(\cal F)|} \cdot b^{\sum_{(x,i)\in V(\cal F)} |\Var(x)|} =\be^{|V(\cal F)|} \cdot b^{\varcount(\cal F)-|V(G)|}.
			$$
			This finishes the proof.
		\end{proof}

		\begin{cory}\label{cory-count}
			Let $G$ be a finite digraph, let $\cal M = (G, \mathbf R,  \pi)$ be a Moser--Tardos tuple, and let  $\De :=\maxdeg(\Rel(G))>0$.  
			\part For $m\in \NN$ the number of grounded $G$-forests $\cal F$ with $|V(\cal F)| = m$ is bounded from above by $(m+1)^{|V(G)|-1}(\me\De)^m$.
			\part Let $\cal F$ be a grounded $G$-forest,  $m:=|V(\cal F)|$, $v:=\varcount(\cal F)$ and $\be:=\beta_{\cal M}$.
			The number of pairs $(\prev,\final)$ such that $(\cal F,\prev,\final)$ is a finalised $\cal M$-landscape is bounded from above by 
			$
			\be^{m}\cdot b^{v}.
			$
			\trap
		\end{cory}
		\begin{proof}
			
			\part Let $T_\De$ be the set of all iso-classes of $\De$-labelled trees. 
			For every $x\in V(G)$, let us fix an order on $N_{\Rel(G)}(x,1)$. This naturally defines a labelling of the edges going out from any given vertex of $\Canvas(G)$ by the elements of $\{0,\dots,\De-1\}$, and subgraphs inherit this $\De$-labelling. This 
			way, since we only consider grounded landscapes, in order to specify 
			$\cal F$ we need to specify a function $f\colon V(G) \to T_\De$ such 
			that $\sum_{x\in V(G)} |V(f(x))| =m$. By Lemma~\ref{lem-count-dtrees}, the 
			number of such functions $f$ is bounded from above by 
			$(m+1)^{|V(G)|-1}(\me\De)^m$.
			
			\part This follows from Lemma~\ref{lem-count-landscapes} and the fact that $\final \in b^{V(G)}$.\qed
			\trap
			\renewcommand{\qedsymbol}{}
		\end{proof}

		\subsection{Interlude: Analysis of the classical MTA}\label{subsec-interludium}
		
		We are now in a position to prove a variant of Theorem~\ref{thm-main} under the assumption that $G$ is finite and  $\pi$ consists of singletons of all vertices in $V(G)$.
		
		The following theorem essentially replicates the known results about the standard MTA. Presenting the proof in this special case will allow us to motivate better the upcoming definitions.
		
		\begin{theorem}\label{thm-main-classic}
			Let $G$ be a finite digraph, let $\cal M = (G, \mathbf R,  \pi)$ be a Moser--Tardos tuple such that $\pi$ consists of the singletons of all vertices of $G$.  Let $\be:=\be_{\cal M}$ as defined in~\eqref{eq:beta}, $\De :=\maxdeg(\Rel(G))>0$ and $d:=\maxoutdeg(G)$. Suppose that $\delta>0$ satisfies
			\begin{equation}\label{eq:delta}
				\be 
				\le  \frac{1}{(\me\De)^{1+\delta}}.
			\end{equation}
			Then for all integers $m\ge 0$ and $k\ge 1$ we have  
			\begin{equation}\label{eq-wqeq35}
				\Pr_{\rnd}\left(|V(\cal L^k_{\cal M, \rnd})| = m\right) \le \frac{ d\,|V(G)|\,(m+1)^{|V(G)|}}{(\me\De)^{\delta\cdot m}}.
			\end{equation}
		\end{theorem}
		
		\begin{myrem}
			\part We note that the event $|V(\cal L^k_{\cal M, \rnd})| = m$ means exactly that the number of clauses resampled in the process of constructing $\MT^{k-1}_{\cal M,\rnd}$ is equal to $m$.
			
			\part The expected total number of resampled clauses is the quantity $\E_{\rnd}(|V(\cal L_{\cal M,\rnd})|)$.
			Note that there is $C=C_{\De,\de, d,b}$ (which does not depend on $k$) such that if we define $m_0:=C\,|V(G)|\log|V(G)|$ then the right-hand side of~\eqref{eq-wqeq35} (which decreases exponentially fast with $m\to\infty$ with the ratio approaching $(\me \De)^{-\de}<1$) is at most $1$ for $m=m_0$ and, for each $m\ge m_0$, decreases by factor at least $1-1/C$ when we increase $m$ by~$1$. It follows that the expectation $\E_{\rnd}(|V(\cal L^k_{\cal M,\rnd})|)$ is at most $O(m_0)$. As such we deduce that $\E_{\rnd}(|V(\cal L_{\cal M,\rnd})|)$ is also bounded by $O(m_0)=O(|V(G)|\log|V(G)|)$.
			\trap
		\end{myrem}
		\begin{proof}[Proof of Theorem~\ref{thm-main-classic}] Take any $k\in \NN_+$. Recall that $\cal L_{\cal M,\rnd}$ is the $\cal M$-landscape associated to $(\cal M,\rnd)$.  Note that by our assumption on $\pi$, the probability space of which 
			$\rnd$ is an element can be taken to be $b^{V(G)\times \NN}$. 
			
			\begin{claim*}  
				Let $\cal L$ be a finalised grounded $\cal M$-landscape of height at most $k-1$. The number of functions $\phi\colon V(G)\to b^{\oplus \NN}$ such that $(\cal L,\phi) = (\cal L^k_{\cal M,\rnd}, \unused^k_{\cal M, \rnd})$ for some $\rnd\in b^{V(G)\times k}$ is bounded from above by 
				$$
				b^{k|V(G)|- \varcount(\cal L)}.
				$$
			\end{claim*}
			\begin{proof}[Proof of Claim]
				Let $\phi$ come from some $\rnd$. Then, for all $x\in V(G)$, we have that $\phi(x)$ is a sequence of length 
				$$ 
				\len(\unused^k_{\cal M,\rnd}(x)) = k - \len(\used_{\cal L^k_{\cal M,\rnd}}(x)) = k -\len(\used_{\cal L}(x)).
				$$
				If we sum this over all $x\in V(G)$ then we obtain $k|V(G)|-\varcount(\cal L)$ by Lemma~\ref{lem-varcount}. Now the claim follows from Remark~\ref{rem-lengths}.
			\end{proof}
			
			Let us fix $k\in \NN_+$. For $m,v\in \NN$ let
			$$
			P_{m,v}:= \Pr_{\rnd} \left(|V(\cal L^k_{\cal M, \rnd})| = m,\, \varcount(\cal L^k_{\cal M,\rnd})=v\right),
			$$
			and let 
			$$
			P_m :=  \Pr_{\rnd} \left(|V(\cal L^k_{\cal M, \rnd})| = m\right),
			$$
			which is the same as the left-hand side of~\eqref{eq-wqeq35}.
			
			Note 
			that $\rnd_k\in b^{V(G)\times k}$ can be recovered from the 
			pair $(\used^k_{\cal M,\rnd}, \unused^k_{\cal M,\rnd})$, and hence also 
			from the pair $(\cal L^k_{\cal M,\rnd}, \unused^k_{\cal M, \rnd})$.  
			This implies that we have
			
			\begin{align}\label{eq-classic-problem}
					P_{m,v} \le \frac{1}{b^{k|V(G)|}}\cdot |\{(\cal L,\phi)\colon &\text{$\cal L$ is a grounded finalised $\cal M$-landscape,}
					\\
					&|V(\cal L)| = m, \varcount(\cal L)=v,  
					\notag
					\\
					&\text{$(\cal L,\phi) = (\cal L^k_{\cal M,\rnd}, \unused^k_{\cal M, \rnd})$ for some $\rnd\in b^{V(G)\times \NN }$}\}|. 
					\notag
			\end{align}
			
			By Corollary~\ref{cory-count} and the claim above,  this is bounded by 
			$$
			\frac{1}{b^{k|V(G)|}}\cdot (m+1)^{|V(G)|-1}(\me\De)^m  \cdot \beta^m \cdot b^{v} \cdot b^{k|V(G)|- v},
			$$
			which is equal to
			$$
			(m+1)^{|V(G)|-1} \cdot (\be e \De)^m.
			$$
			This bound on $P_{m,v}$ does not depend on $v$. Since $P_{m,v}=0$ when $v>dm+|V(G)|$, we deduce that 
			$$
			P_m \le (dm+|V(G)|)\cdot (m+1)^{|V(G)|-1} \cdot (\be e \De)^m.
			$$
			Now using~\eqref{eq:delta} and that, rather roughly, $dm+|V(G)|\le d(m+1)\,|V(G)|$, we get the required bound.
		\end{proof}

		\subsection{Some ideas behind re-using random bits}\label{se:ideas}
		
		Let us informally describe what needs to be modified in order to prove Theorem~\ref{thm-main}. In Theorem~\ref{thm-main-classic} we had an element $\rnd\in b^{V(G)\times \NN}$, i.e.~each vertex has ``its own stream of randomness''. In Theorem~\ref{thm-main} this is no longer the case since $\rnd\in b^{\pi\times \NN}$, so there are only $|\pi|$  streams of randomness, no matter how big the digraph $G$ is. However the number of landscapes is still exponential in $|V(G)|$, because the number of $G$-forests is exponential in $|V(G)|$.  Since the inequality~\eqref{eq-classic-problem} would start with $\frac{1}{b^{k|\pi|}}$ on the right-hand side, it would not provide any interesting information because the right-hand side would be strictly larger than $1$ (in fact exponentially large in $|V(G)|$).
		
		The remedy is to restrict the landscape to a carefully chosen  subset $U\subset V(G)$  on which $\pi$ induces a singleton partition. The choice of this subset is different for different $\rnd\in b^{\pi\times k}$; essentially, we take a ball $U$ of carefully chosen (bounded) radius $r$
		around one of the most resampled vertices $y$. Although the restriction of the finalised landscape to $U$ does not in general allow us to reconstruct all random bits used at $U$, it does determine the used bits for every $x\in U$ such that all clauses containing $x$ are entirely inside~$U$.    
		By choosing $r$ carefully (using the assumption of subexponential growth), we can ensure that the number of ``boundary'' clauses is much smaller than the number of ``internal'' clauses and, additionally,  the restriction of $\pi$ to $U$ is injective. The main new idea is roughly to restrict the argument of Theorem~\ref{thm-main-classic} to the set $U$ (where no input bit is re-used), using the trivial upper bound of $b$ for the number of choices of each unknown random bit inside $U$ and doing the union bound over all possible restrictions of $G$ to $U$ (where there are only bounded many choices up to a $\pi$-preserving isomorphism).
		
		However, when we restrict a grounded landscape to $U$, it does not need to be grounded anymore, while the bounds in Corollary~\ref{cory-count} are valid only for grounded landscapes. As such we have to find a grounded landscape on $U$ which is equivalent (meaning that the functions that turn landscape decorations into sequences of used bits are the same for the two landscapes). 
		
		The next two subsections of this section are about defining the restrictions of landscapes and finding equivalent grounded landscapes, respectively. After discussing these preliminaries, we will prove Theorem~\ref{thm-main} in the third remaining subsection.

		\subsection{Landscape restrictions}\label{subsec-restrictions}
		\newcommand{\Res}{\operatorname{Res}}
		
		
		Let $\cal M = (G, \mathbf R,  \pi)$ be a Moser--Tardos tuple and let $\cal L=(\cal F,\prev)$ be an $\cal M$-landscape. We say that $U\subset V(G)$ is \emph{$\pi$-unique} if the map $S_\pi\colon V(G)\to \pi$ is injective on $U$. 
		
		The aim of this subsection is to define restrictions $\Res_U(\cal M)$ and $\Res_U(\cal L)$ of  $\cal M$ and $\cal L$ respectively, to a $\pi$-unique set $U$.
		\begin{myrem} Some extra effort is needed because it is convenient to define e.g.~the restriction $\Res_U(\cal M)$ in such a way that the underlying graph always has the same set of vertices, i.e.~$\pi$. 
		\end{myrem}
		
		For a digraph $G$ and a set $U\subset V(G)$ we define $G\restriction U := (U, E(G)\cap U^2)$ to be the digraph induced from $G$ on $U$.
		
		Let $\cal M = (G, \mathbf R,  \pi)$ be a Moser--Tardos tuple and let 
		$U\subset V(G)$ be $\pi$-unique. If $\al\in \pi$ is such that $\al = S_\pi(x)$ for some $x\in U$ then we let $T_U(\al):=x$. Thus $T_U$ is a bijection from a subset of $\pi$ to~$U$.
		
		We first define the Moser--Tardos tuple $\Res_U(\cal M) = (G',\mathbf R', \pi')$ as follows. We let 
		$V(G'):=\pi$ and $E(G'):=\{(S_\pi(u),S_\pi(v)):(u,v)\in E(G),\ u,v\in U\}$, that is, the set of edges of $G'$ is the minimal subset of 
		$\pi\times \pi$ which makes the map $S_\pi$ a digraph isomorphism from 
		$G\restriction U$ onto its image in $G'$. We define $\pi'$ to be the partition of the set $\pi$ into singletons.
		
		Let $\al\in \pi$. We define $\mathbf R'(\al)$ by considering two cases. 
		
		First, suppose that $\al= S_\pi(x)$ for some $x\in U$ and furthermore $\Var(x)\subset U$. Then we let 
		\begin{equation}\label{eq-sdg2}
			\mathbf R' (\al) := \{f\circ T_U\colon f \in \mathbf R(x)\}.
		\end{equation}
		Note that when $f\in \mathbf R(x)\subset b^{\Var_G(x)}$ then $f\circ T_U$ is defined on $\Var_{G'}(\al)$, and thus in the definition above the function $f\circ T_U$ should be understood as an element of $b^{\Var_{G'}(\al)}$. In other words, the composition operator in~\eqref{eq-sdg2} should be understood as acting on partial functions.
		
		If the first case does not hold then we define $\mathbf R'(\al)$ to be $b^{\Var_{G'}(\al)}$, that is, the clause $\alpha$ does not impose any constraints.
		
		Clearly, $|\mathbf R'(\al)|/b^{|\Var_{G'}(\al)|}\ge |\mathbf R(x)|/b^{|\Var_G(x)|}$ for every $x\in U$ and $\al=S_\pi(x)$, so  $\beta_{(G', \mathbf R')}\le \beta_{(G, \mathbf R)}$. Informally speaking, we induce the local rule $\mathbf R'$ as the restriction $\mathbf R$ to $U$ where it makes sense, and put no constraint otherwise.
		
		We proceed similarly when $\cal L=(\cal F,\prev)$ is an $\cal 
		M$-landscape. Namely, we let $\Res_U(\cal L)$ be the following $\Res_U(\cal 
		M)$-landscape $(\cal F', \prev')$. We let 
		$$
		V(\cal F') := \{(S_\pi(x),i)\colon (x,i)\in V(\cal F),\ x\in U \text{ and } \Var(x)\subset U\},
		$$ 
		and we add an edge from $(S_\pi(x),i)$ to $(S_\pi(y),i+1)$ if and only if both $(S_\pi(x),i)$ and $(S_\pi(y),i+1)$ belong to $V(\cal F')$ and there is an edge from $(x,i)$ to $(y,i+1)$ in $\cal F$. For $(\al,i)\in V(\cal F')$, take the (unique) $x\in U$ with $\al = S_\pi(x)$ and let 
		\begin{equation}\label{gg21}
			\prev'(\al,i) := \prev(x,i)\circ T_U.
		\end{equation}
		We note that the composition in~\eqref{gg21} should be understood as acting on partial functions, similarly as in the case of the definition of $\mathbf R'(\alpha)$.
		Clearly, if $(\al,i)\in V(\cal F')$ then $\prev'(\al,i)$ violates $\mathbf R'(\al)$, as formally required in our definition of a landscape.
		
		Finally, suppose that we also have a function $\final:V(G)\to b$ (so that  $(\cal F,\prev,\final)$ is a finalised landscape).  Its restriction $\Res_U(\final)$ is defined to  assume value $\final(x)$ on every $\alpha=S_\pi(x)$ with $x\in U$ and value $0$ on all other elements of~$\pi$. Also, we let
		\[
		\Res_U(\cal F,\prev,\final):=(\cal F',\prev',\Res_U(\final)).
		\]
		
		\begin{lemma}\label{lem-restrict}
			Let $\cal M = (G, \mathbf R,  \pi)$ be a Moser--Tardos tuple, let $\cal L$ be a finalised $\cal M$-landscape, let $U\subset V(G)$ be a $\pi$-unique set, and let $x\in U$ be such that $\Var_G(\Cl_G(x))\cup \Cl_G(x)\subset U$. Then 
			$$
			\used_{\cal L}(x) = \used_{\Res_U(\cal L)} (S_\pi(x)).
			$$ 
		\end{lemma}
		\begin{proof}
			The condition $\Var(\Cl(x))\cup \Cl(x)\subset U$ ensures that whenever $x$ is a variable in some clause $y$ then $y\in U$ and $\Var(y) \subset U$. This means that for all $y\in \Cl_G(x)$ we have that $\mathbf R' (S_\pi(y))$ is defined by the formula in~\eqref{eq-sdg2}. Similarly for all $i\in\NN$ with $(y,i)\in V(\cal L)$ we have that  $\prev'(S_\pi(y),i)$ is defined by the formula in~\eqref{gg21}. Thus the required equality follows from the definitions of $\used$ and $\Res_U(\cal L)$. 
		\end{proof}

		\subsection{Equivalent landscapes}\label{subsec-equiv}
		\newcommand{\G}{\operatorname{G}}
		\newcommand{\Push}{\operatorname{Push}}
		
		Let $\cal M = (G, \mathbf R,  \pi)$ be a Moser--Tardos tuple, and let $\cal K$ and $\cal L$ be finalised $\cal M$-landscapes. We say that $\cal L$ and $\cal K$ are \emph{equivalent} if $|V(\cal L)| = |V(\cal K)|$ and $\used_{\cal L} = \used_{\cal K}$. In particular, if $\cal L$ and $\cal K$ are equivalent then $\varcount(\cal L) = \varcount(\cal K)$, since  $\varcount(\cal L) = \len(\used_{\cal L})$.
		
		The purpose of this subsection is to prove the following proposition.
		
		\begin{proposition}\label{prop-equiv}
			Let $\cal M = (G, \mathbf R,  \pi)$ be a Moser--Tardos tuple with finite $G$, and let $\cal L$ be a finalised $\cal M$-landscape. There exists a grounded finalised landscape $\cal K$ which is equivalent to $\cal L$.
		\end{proposition}
		
		The proof is quite straightforward, but it is somewhat fiddly to set up the induction. 
		
		\begin{myrem} Proposition~\ref{prop-equiv} holds true also if we remove the assumption that $G$ is finite. However, let us stress that we do not need this greater generality for the proof of Theorem~\ref{thm-main}.
		\end{myrem}
		
		If $\cal L$ is a landscape and $\tau\in \Trees(\cal L)$ is not grounded, then let us say that $\tau$ is \emph{airborne}, and let $\Trees^\vee(\cal L)$ be the set of all airborne trees.
		Clearly in order to prove Proposition~\ref{prop-equiv},  it is enough to prove the following lemma.
		
		\begin{lemma}
			Let $\cal M = (G, \mathbf R,  \pi)$ be a Moser--Tardos tuple with finite $G$. If $\cal L$ is a finalised $\cal M$-landscape with
			$|\Trees^\vee(\cal L)|>0$ then there exists an equivalent finalised landscape $\cal L'$ 
			such that  $|\Trees^\vee(\cal L')| < |\Trees^\vee(\cal L)|$.
		\end{lemma}
		
		\newop{fly}{fly}
		\begin{proof}  
			Suppose by way of contradiction that the lemma is false. Let us first define $A_1$ to be the set of all finalised $\cal M$-landscapes which are not equivalent to a grounded one. Let $m_1:=\min\{ |\Trees^{\vee}(\cal L)|\colon \cal L \in A_1\}$ and let 
			$$
			A_2 := \{\cal L \in A_1\colon |\Trees^{\vee}(\cal L)|=m_1\}.
			$$
			
			Let $m_2 :=\min\{|V(\tau)|\colon \tau \in \Trees^{\vee}(\cal L) \text{ for some $\cal L\in A_2$}\}$ and let $\cal L=(\cal F, \prev,\final)$ belong to $A_2$ and $\tau\in \Trees^\vee(\cal L)$ be such that  $|V(\tau)|=m_2$. In other words, we pick a  counterexample $\cal L$ with the smallest number of airborne trees and, among all such $\cal L$, one that minimises the size of  an airborne tree~$\tau$.
			
			The proof has 3 steps/cases. 
			
			\part[label=\textbf{(Step 1)}] This step is applicable in the case that we can ``push $\tau$ down''. This is illustrated in Figure~\ref{fig-lala}. 
			
			\begin{figure}[h]
				\begin{subfigure}{1\textwidth}
					\centering
					\scalebox{0.4}{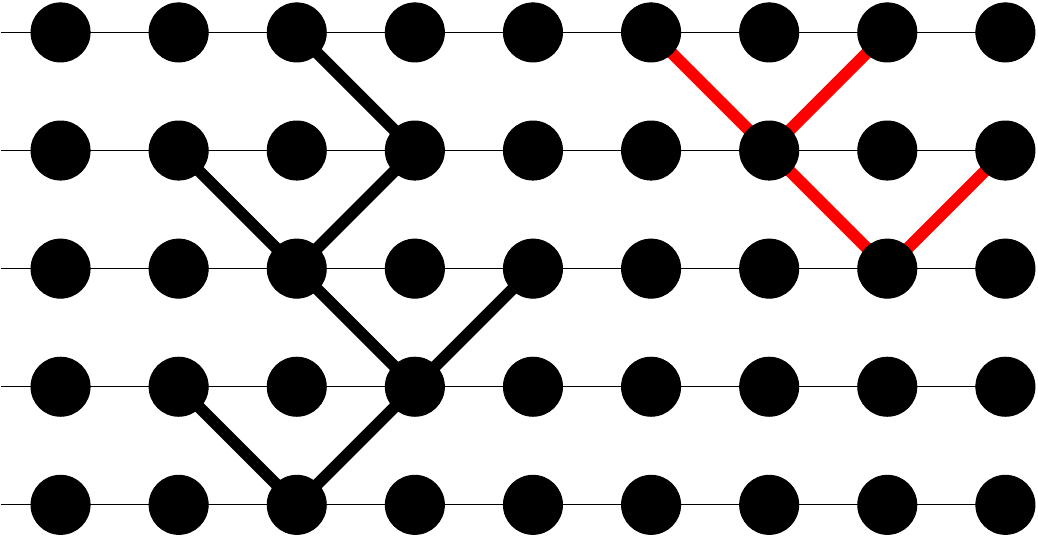}
					\caption{Forest with two components, $\tau$ is red. The horizontal edges are the edges in $\Rel(G)$. We can push $\tau$ down
						one level because after pushing the vertices of the forest form independent sets at each level.}
				\end{subfigure}
				\newline
				\begin{subfigure}{1\textwidth}
					\centering
					\scalebox{0.4}{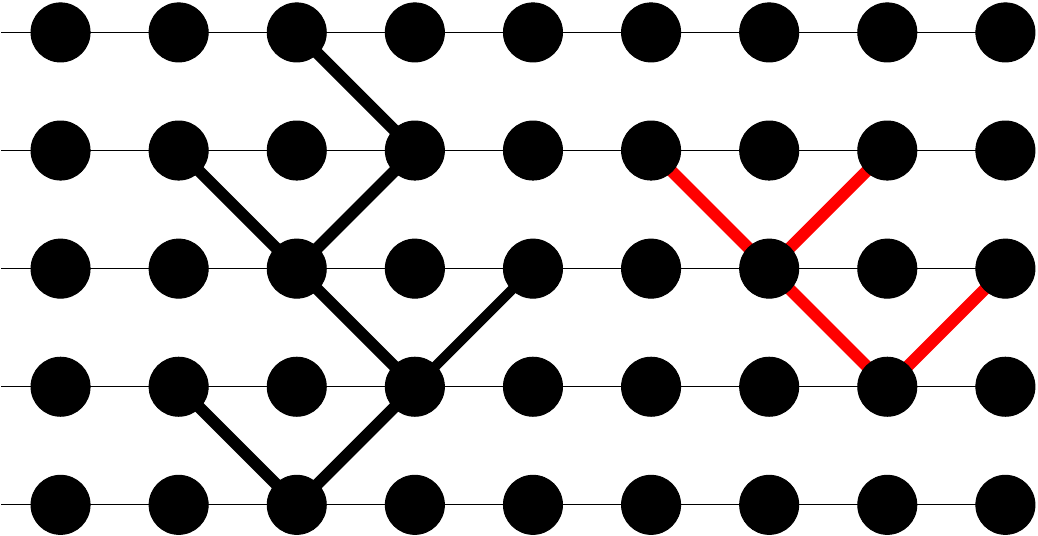}
					\caption{The modified forest with $\tau$ pushed down by one level. We cannot push $\tau$ any further down because the vertices in the middle row would not form an independent set.}
				\end{subfigure}
				\caption{}\label{fig-lala}
			\end{figure}
			
			Formally, we say that $\tau$ \emph{can be pushed by one level} if 
			for all $(x,i)\in V(\tau)$ and $ (y,i-1)\in V(\cal F)\setminus 
			V(\tau)$ we have $\Var_G(x)\cap \Var_G(y)=\emptyset$. If this is the 
			case then we can define an independent $G$-forest $\cal F'$ with 
			\begin{align*}
				V(\cal F') &= \left(V(\cal F)\setminus V(\tau)\right) \cup \{(x,i-1)\colon (x,i)\in V(\tau)\},
				\\
				E(\cal F') &= \left(E(\cal F)\setminus E(\tau)\right) \cup \{((x,i-2),(y,i-1))\colon ((x,i-1),(y,i))\in E(\tau)\}.
			\end{align*}
			
			We define $\cal L'=(\cal F',\prev',\final)$, where $\prev'(x,i):=\prev(x,i)$ when $(x,i) \in V(\cal F)\setminus V(\tau)$ and $\prev'(x,i):= \prev(x,i+1)$ when $(x,i+1)\in V(\tau)$, with these two cases being exclusive since we can push $\tau$ down.
			
			The fact that $\used_{\cal L} = \used_{\cal L'}$ follows from the fact that, informally speaking we have not changed the order of any variables which appear in $\cal F$ and $\cal F'$. More precisely, recall that to define $\used_{\cal L}(x)$ we list those  vertices  $(y_1,i_1),\ldots, (y_t,i_t)$  of $V(\cal F)$  such that $x\in \Var(y_j)$, ordered by the second coordinate, and we let $\used_{\cal L}(x)$ be the sequence $\prev(y_1,i_1)(x),\ldots, \prev(y_t,i_t)(x),\final(x)$. 
			
			Now  in order to compute $\used_{\cal L'}(x)$ we get a list $(y_1, i_1'),\ldots, (y_t,i_t')$ in $V(\cal F')$, where each $i_m'$ is $i_m$ or $i_m-1$. Since no two of $i_1',\dots,i_t'$ are the same, we have $i_1'<\ldots < i_t'$ and the clauses $y_1,\dots,y_t$ come in the same order as for $\cal L$. This shows the desired equality $\used_{\cal L} = \used_{\cal L'}$.
			
			Thus after repeatedly passing to an equivalent landscape if necessary, we can assume that $\tau$ cannot be pushed down by one level. This means that there exist $(x,i)\in V(\tau)$ and $(y,i-1)\in V(\cal F)\setminus V(\tau)$ such that $\Var(x)\cap \Var(y)\neq \emptyset$. We take one such triple $(x,y,i)$ and do one of the following two steps.
			
			\item[\textbf{(Step 2)}] In this step let us assume that $x$ is not the root of $\tau$. In this case in order to define $\cal L'$ we do not modify $V(\cal F)$ at all, we only remove the unique incoming edge at $x$, and add the edge $((y,i-1), (x,i))$ to the forest (see Figure~\ref{fig-wew}).
			
			\begin{figure}[h]
				\centering
				\scalebox{0.4}{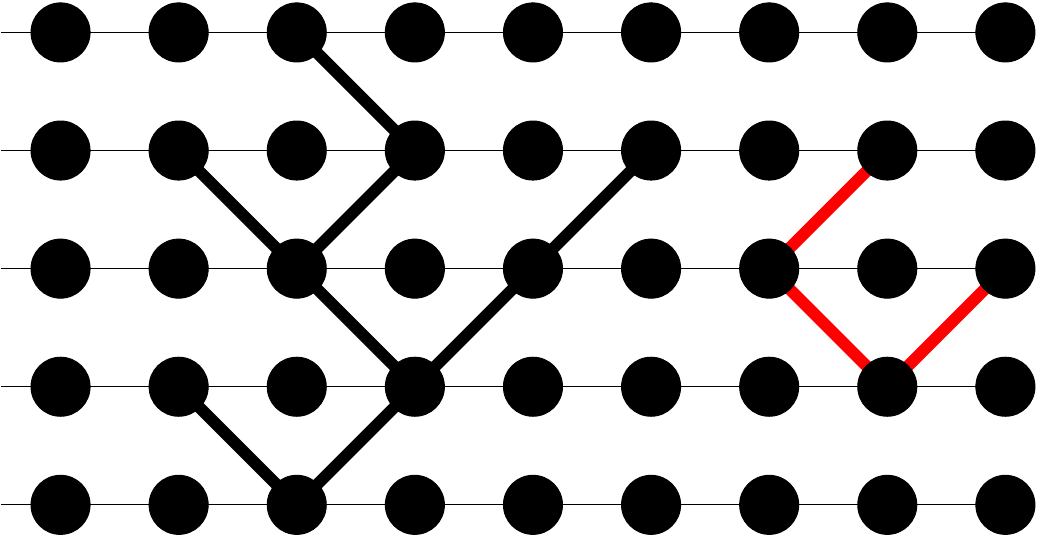}
				\caption{The forest from Figure~\ref{fig-lala}(b), after the operation described in (Step 2)} 
				\label{fig-wew}
			\end{figure}
			Since we do not modify $V(\cal F)$, $\prev$ and $\final$, it is clear that $\cal L'$ is equivalent to $\cal L$. Also, the number of airborne trees is the same. This is a contradiction with the definition of $m_2$, since in $\cal L'$ we have an airborne tree with less than $m_2$ vertices.
			
			\item[\textbf{(Step 3)}] Now let us assume that $x$ is the root of $\tau$. This is illustrated in Figure~\ref{fig-pp}.
			
			\begin{figure}[h]
				\begin{subfigure}{1\textwidth}
					\centering
					\scalebox{0.4}{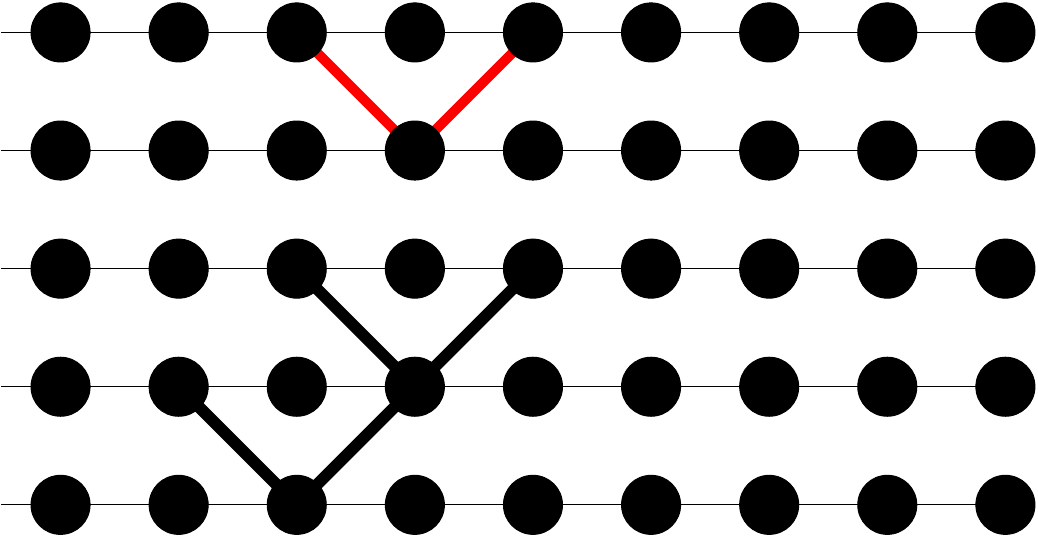}
					\caption{Forest with two components, $\tau$ is red.  We cannot push $\tau$ down
						one level, because afterwards the vertices in the middle row would not form an independent set.}
				\end{subfigure}
				\newline
				\begin{subfigure}{1\textwidth}
					\centering
					\scalebox{0.4}{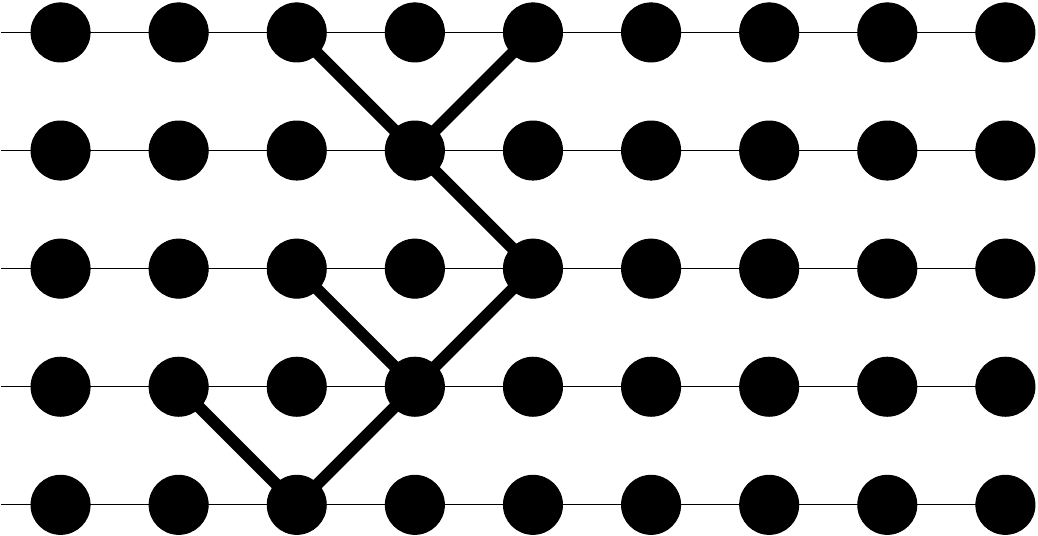}
					\caption{The modified forest after the operation described in (Step 3), i.e.~we join the root of $\tau$ to another vertex of $\cal F$}
				\end{subfigure}
				\caption{}\label{fig-pp}
			\end{figure}
			
			In this case the only modification we make to define $\cal L'$ is adding the edge  $((y,i-1),(x,i))$ to the forest. As in the previous step, it is clear that $\cal L'$ is equivalent to $\cal L$. But this is a contradiction with the definition of $m_1$, because $\cal L'$ has one less airborne tree than $\cal L$. This final contradiction finishes the proof of the lemma, and hence also of Proposition~\ref{prop-equiv}.\qedhere
			\trap
		\end{proof}

		\subsection{Proof of Theorem~\ref{thm-main}}\label{subsec-mainproof}
		
		The assumption that $G\in \subexp(\nn,\eps,d)$ in Theorem~\ref{thm-main} will be used via the following lemma.

		\begin{lemma}\label{lem-growth} Let $G$ be a digraph, let $\eps>0$ and let us assume that $\nn\in\NN_+$ is such that for all $x\in V(G)$ we have $|N(x,3\nn)|\le (1+\eps)^{\nn}$. Let $h\colon V(G) \to \N$  be a bounded function, and let $y\in V(G)$ be such that $h(y) = \max_{x\in V(G)} h(x)$. Then there exists $r\in\{3,\ldots, 3\nn\}$ such that 
			\beq[jaba]
			\sum_{x\in N(y,r)} h(x)  \le (1+\eps)\cdot \sum_{x\in N(y,r-3)} h(x).
			\eeq
		\end{lemma}
		
		\begin{proof}
			Suppose that for all $r\in \{3,\ldots,3\nn\}$ the inequality \eqref{jaba} does not hold. Then 
			$$
			\sum_{x\in N(y,3\nn)} h(x) > 
			(1+\eps)\sum_{x\in N(y,3\nn-3)} h(x) >\ldots > (1+\eps)^{\nn} \sum_{x\in N(y,0)} h(x)= (1+\eps)^{\nn}\cdot h(y).
			$$
			
			Since $|N(y,3\nn)|\le (1+\eps)^{\nn}$, we deduce that for some $z\in N(y,3\nn)$ we have $h(z)>h(y)$, which contradicts the maximality of $h(y)$.
		\end{proof}

		Theorem~\ref{thm-main} is a consequence of the following theorem (the difference is essentially that in Theorem~\ref{thm-main} we only consider $h^\infty_{\cal M,\rnd}$). Indeed, for any fixed integer $m$, the events $\{\max_{x\in V(G)} h^k(x)>m\}$ are increasing with $k$ and their union over all $k\in\NN_+$ is the event $\{\sup_{x\in V(G)} h^\infty(x)>m\}$.
		
		\begin{theorem}\label{thm-main-analysis}
			Let $\cal M=(G,\mathbf R, \pi)$ be a Moser--Tardos tuple, let  $\be:=\be_{\cal M}$ as defined in~\eqref{eq:beta} and let $\De:= \maxdeg(\Rel(G))>0$. Suppose that integers $\nn,d\in \NN_+$ and reals $\eps,\delta>0$ satisfy that 
			\part $G\in \subexp(\nn,\eps,d)$,
			\part the partition $\pi$ is $3\nn$-sparse, and  
			\part it holds that
			\begin{equation}\label{eq:delta2}
				\be\le \frac{1}{(\me\De)^{1+\delta}\, b^{\eps d}}.
			\end{equation}
			\trap
			Then there is a constant $K$ which depends only on $b$, $\de$, $d$ and $|\pi|$ such that for all $m,k\in \NN_+$ we have  
			\begin{equation}\label{eq-wqeq}
				\Pr_{\rnd}\left(\max_{x\in V(G)} h^k_{\cal M,\rnd}(x)> m\right) \le \frac{K(m+1)^{|\pi|}} {(\me\De)^{\de m}}.
			\end{equation}
		\end{theorem}
		
		\begin{proof}
			Take any  $k\in \NN_+$. We start by associating to each $\rnd\in b^{\pi\times k}$ a Moser--Tardos tuple $\cal N_{\rnd}$ and an $\cal N_{\rnd}$-landscape $\cal K_{\rnd}$, in such a way that some of the random bits in $\rnd$ can be recovered from $\cal K_{\rnd}$. We do it as follows.
			
			Let us fix $\rnd\in b^{\pi\times k}$. Let $h := h^k_{\cal M,\rnd}$. We recall that $h(x)$ is the total number of resamplings of the variable $x$ which the algorithm makes on the input $(\cal M,\rnd)$  when defining $\MT^0,\ldots, \MT^{k-1}$. 
			
			Fix one point $y=y_{\rnd}\in V(G)$ where $h$ attains its maximum. Now by Lemma~\ref{lem-growth} we obtain $r=r_{\rnd}$ in $\{3,\ldots, 3\nn\}$ such that 
			$$
			\sum_{x\in N_G(y,r)} h(x)  \le (1+\eps)\cdot \sum_{x\in N_G(y,r-3)} h(x).
			$$
			From here we deduce that
			\begin{equation}\label{eq-ssg}
				\sum_{x\in N_G(y,r-3)} h(x)\ge \frac1{1+\eps}\sum_{x\in N_G(y,r)} h(x)= (1-\eta) \sum_{x\in N_G(y,r)} h(x),
			\end{equation}
			where we define $\eta:=\eps/(1+\eps)$. Clearly, $0<\eta<\min\{\eps,1\}$.
			
			Note that $N_G(y,r)$ is $\pi$-unique since $\pi$ is $3\nn$-sparse. We let $\cal N_{\rnd}$ be the Moser--Tardos tuple $\Res_{N_G(y,r)}(\cal M)$, and we let $\cal K_{\rnd}$ be a grounded finalised  $\cal N_{\rnd}$-landscape equivalent to $\Res_{N_G(y,r)}(\cal L^k_{\cal M, \rnd})$. The fact that we can take $\cal K_{\rnd}$ to be grounded follows from Proposition~\ref{prop-equiv}.
			
			We will shortly show that for all $m\in \NN$ we have 
			\begin{equation}\label{eq-willdo}
				\Pr_{\rnd}\left( |V(\cal K_{\rnd})| \ge m\right) \le \frac{K\cdot (m+1)^{|\pi|}} {(\me\De)^{\de m}},
			\end{equation}
			where $K$ depends only on $b$, $\de$, $d$ and $|\pi|$.  Let us first explain how this is enough to establish~\eqref{eq-wqeq}.  
			
			Recall that we fixed $y=y_{\rnd}$ to be a vertex of $G$ that maximises~$h$. Note that 
			if $x\in N_G(y,r-3)$, then every clause containing $x$ lies within distance at most $r-2$ from $y$, and every variable in such a clause lies within distance at most $r-1$ from \(y\); hence the hypothesis of Lemma~\ref{lem-restrict} holds inside $N_G(y,r)$.
			Thus for $x\in N_G(y,r-3)$, by Lemma~\ref{lem-equal-used}, we also have
			\begin{equation}\label{sp39}
				\used_{\cal K_{\rnd}}(S_\pi(x)) = \used^k_{\cal M,\rnd}(x).
			\end{equation}
			Since $r\ge 3$, we deduce in particular that  
			$$
			|V(\cal K_{\rnd})| +1 \ge \len(\used_{\cal K_{\rnd}}(S_\pi(y)))  = \len(\used^k_{\cal M,\rnd}(y)) = h(y),$$
			where $+1$ accounts for the last symbol of $\used_{\cal K_{\rnd}}(S_\pi(y))$ coming from the $\final$-part of $\cal K_{\rnd}$.
			Thus we have
			\begin{align*}
				\Pr_{\rnd}\left(\max_{x\in V(G)} h(x)> m\right) 
				= \Pr_{\rnd}\left(h(y)> m\right)
				& \le \Pr_{\rnd}\left(|V(\cal K_{\rnd})| \ge m\right),
			\end{align*}
			which establishes~\eqref{eq-wqeq} when~\eqref{eq-willdo} holds. Thus the rest of the proof is devoted to establishing~\eqref{eq-willdo}.
			
			For $m,v\in \NN$, let
			$$
			P_{m,v}:= \Pr_{\rnd} \left(|V(\cal K_{\rnd})| = m,\, \varcount(\cal K_{\rnd})=v\right),
			$$
			and let 
			$$
			P_m :=  \Pr_{\rnd} \left(|V(\cal K_{\rnd})| = m\right).
			$$
			
			
			For $\al \in \pi$ such that $\al = S_\pi(x)$ for some  $x\in N_G(y,r)$, we have that $\len(\used_{\cal K_{\rnd}}(\al)) \le h(x)$, and if furthermore $x\in N_G(y,r-3)$ then $\len(\used_{\cal K_{\rnd}}(\al)) = h(x)$. Thus the inequality~\eqref{eq-ssg} shows that  
			$$
			\len_{S_\pi(N_G(y,r-3))}(\used_{\cal K_{\rnd}})+|\pi|>(1-\eta)\len(\used_{\cal K_{\rnd}}),
			$$
			where the term $|\pi|$ is a strict upper bound on the total contribution from the  elements of $\pi\setminus S_\pi(N_G(y,r))$ (whose streams contribute only one bit each).
			We define $\C_{\rnd}\colon \pi \to b^{\oplus \NN}$ as $\C_{\rnd}(\al) := \unused^k_{\cal M,\rnd}(x)$ 
			when $\al = S_\pi(x)$ for some $x \in N_G(y,r-3)$, and $\C_{\rnd}(\al): = \rnd_k(\al,\ast)$ otherwise. We note that 
			$$
			\len(\C_{\rnd}(\al)) + \len(\used_{\cal K_{\rnd}}(\al)) =k
			$$
			for all $\alpha\in S_\pi(N_G(y,r-3))$, and $\len(\C_{\rnd}(\al)) = k$ for all remaining $\al\in\pi$. 
			
			Note that, because of~\eqref{sp39}, we can recover $\rnd_k$ from the pair $(\used_{\cal K_{\rnd}}, \C_{\rnd})$,
			and hence also from the triple $(\cal N_{\rnd}, \cal K_{\rnd}, \C_{\rnd})$. We deduce that 
			\begin{align*}
				P_{m,v} \le \frac{1}{b^{k|\pi|}}\cdot |\{&(\cal N, \cal K,\C)\colon \text{$\cal N$ is a Moser--Tardos tuple such that $\beta_{\cal N}\le \beta$,}
				\\
				&\text{the underlying digraph $G_{\cal N}$ has vertex set $\pi$,}
				\notag
				\\
				&\text{$\maxdeg(\Rel(G_{\cal N}))\le \De$, $\maxoutdeg(G_{\cal N})\le d$,}
				\notag
				\\
				&\text{$\cal K$ is a grounded finalised $\cal N$-landscape with }
				|V(\cal K)| = m\text{ and } \varcount(\cal K)=v,  
				\notag
				\\
				&\text{$\cal C\colon \pi \to b^{\oplus\NN}$ is $k$-complementary to $\used_{\cal K}$ on a set $Y\subset \pi$ such that}
				\notag
				\\
				&\len_Y(\used_{\cal K}) +|\pi|>(1-\eta)\len(\used_{\cal K}), \text{ and }
				\notag
				\\
				&\len(\cal C(\al)) = k \text{ for } \al\in \pi\setminus Y \}|.
				\notag
			\end{align*}
			
			Let us estimate the number of the triples $(\cal N, \cal K, \C)$ which are as described above.
			
			The Moser--Tardos tuple $\cal N$ is given by a digraph with $\maxoutdeg$ at most $d$, whose vertex set is $\pi$, and a local rule on it. Rather roughly, the number of possibilities for the digraph is at most $(\sum_{j=0}^d \binom{|\pi|}{j})^{|\pi|}\le (|\pi|+1)^{d|\pi|}$. Also, the number of possibilities for the local rule on a digraph with  $|\pi|$ vertices is at most $2^{b^d\cdot |\pi|}$. Thus all in all there are at most 
			$$
			(|\pi|+1)^{d|\pi|} \cdot 2^{b^d\cdot |\pi|} = \left((|\pi|+1)^d2^{b^d}\right)^{|\pi|}
			$$
			possibilities for $\cal N$.

			Let us estimate the number of possibilities for the pair $(\cal K,\cal C)$ when the Moser--Tardos tuple $\cal N$ is fixed.

			By Lemma~\ref{lem-count-dtrees} and Corollary~\ref{cory-count}, there are at most
			$$
			(m+1)^{|\pi|-1}(\me\De)^m \cdot 
			\beta_{\cal N}^m\cdot b^{v}\le 
			(m+1)^{|\pi|-1}(\me\De)^m \cdot 
			\beta^m\cdot b^{v}
			$$
			possibilities for the grounded finalised $\cal N$-landscape~$\cal K$. By~\eqref{eq:delta2} and $\eta<\eps$, 
			this is at most
			\begin{equation}\label{wgf234}
				(m+1)^{|\pi|-1}\cdot \frac{1}{(\me\De)^{\delta m}} \cdot b^{v -\eta  d m}.
			\end{equation}
			Since $v \le |\pi|+dm$, we can bound~\eqref{wgf234} from above by 
			$$
			(m+1)^{|\pi|-1}\cdot \frac{1}{(\me\De)^{\delta m}} \cdot b^{(1-\eta)v} \cdot b^{\eta|\pi|}.
			$$
			
			Next, let us count the possibilities for $\C$. There are $2^{|\pi|}$  possibilities for $Y$. For $\alpha\in \pi\setminus Y$ we have $\len(\C(\alpha)) =k$, and for $\alpha\in Y$ we have $\len(\C(\alpha)) = k-\len(\used_{\cal K}(\alpha))$. Thus the choice of $Y$ determines $\len(\C(\alpha))$ for all $\alpha\in \pi$. Furthermore, we have
			$$
			\len(\C) = \sum_{y\in \pi} \len(\C(y)) =k|\pi| - \len_Y(\used_{\cal K})
			$$
			and $\len_Y(\used_{\cal K})+|\pi|>(1-\eta)\len(\used_{\cal K})$, which by Lemma~\ref{lem-varcount} is equal to $(1-\eta)v$. 
			Therefore we have 
			$$
			\len(\C) \le 
			(k+1)|\pi|- (1-\eta)v.
			$$
			Now, Remark~\ref{rem-lengths} shows that  there are at most 
			$$
			2^{|\pi|} \cdot  b^{(k+1)|\pi|-(1-\eta)v}   
			$$
			possibilities for $\C$.  
			
			Thus we see that 
			\begin{align*}
				P_{m,v} &\le ((|\pi|+1)^d\, 2^{b^d}\, b)^{|\pi|} \cdot   b^{\eta|\pi|} \cdot (m+1)^{|\pi|-1}\, \frac{1}{(\me\De)^{\de m}}
				\cdot 2^{|\pi|}.
			\end{align*}
			Since this bound does not depend on $v$, and $P_{m,v}=0$ when $v>dm+|\pi|$, we can estimate $P_m$ from above by multiplying the above bound by $dm+|\pi|+1\le (d+|\pi|)(m+1)$. Furthermore we bound $b^{\eta|\pi|}< b^{|\pi|}$. So we deduce that 
			$$
			P_m \le (d+|\pi|)\cdot ((|\pi|+1)^d\,2^{b^d+1}\,b^{2})^{|\pi|} \cdot    (m+1)^{|\pi|}\, \frac{1}{(\me\De)^{\de m}}.
			$$
			Recall that for $x\in (0,1)$ and $t\in \NN_+$ we have 
			$$
			\left(\frac{1}{1-x}\right)^t = \sum_{i=0}^\infty \frac{(i+t-1)\cdot \ldots\cdot (i+1)}{(t-1)!}\, x^i.
			$$
			It follows that  
			\begin{align*}
				\sum_{j\ge m} (j+1)^{t-1} x^j &= x^m \sum_{i\ge 0} (i+m+1)^{t-1} x^i 
				\\
				&\le (m+1)^{t-1} x^m \sum_{i\ge 0} (i+1)^{t-1} x^i 
				\\
				&\le (m+1)^{t-1} x^m \frac{(t-1)!}{(1-x)^t}.
			\end{align*}
			Using $t :=|\pi|+1$ and $x:= (\me\De)^{-\de}$ gives us the inequality 
			$$
			\sum_{j\ge m} (j+1)^{|\pi|} \cdot  \frac{1}{(\me\De)^{\de j}} \le \frac{(m+1)^{|\pi|}}{(\me\De)^{\de m}} \cdot \frac{|\pi|!}{(1-(\me\De)^{-\de})^{|\pi|+1}}.
			$$
			Therefore, we deduce using $\De\ge 1$ that
			$$
			\Pr_{\rnd} \left(|V(\cal K_{\rnd})| \ge m\right) = \sum_{j \ge m} P_j \le \frac{ (d+|\pi|)((|\pi|+1)^d\,2^{b^d+1}\,b^{2})^{|\pi|}(|\pi|!)}{(1-\me^{-\de})^{|\pi|+1}}\cdot  \frac{(m+1)^{|\pi|}}{(\me\De)^{\de m}}.
			$$
			Note that the first factor in the right-hand side does not depend on $m$ (nor on $k$).  This establishes~\eqref{eq-willdo} and finishes the proof.
		\end{proof}
		\section{Borel version of the Lov\'asz Local Lemma}\label{sec-borel}
		
		In this section, we derive a Borel version of the Lov\'asz Local Lemma (Theorem~\ref{thm-main-borel} here). A derivation of Theorem~\ref{thm-main-borel} from Theorem~\ref{thm-main} is rather routine for experts; however, we present some details for the sake of completeness.

		Let $G= (V,\mathcal B,E)$  be a \emph{Borel digraph}, that is, $(V,E)$ is a digraph, $(V,\mathcal B)$ is a standard Borel space and the edge set $E$ is Borel as a subset of $V\times V$. (We refer the reader to~\cite{Pikhurko21bcc} for all omitted 
		definitions.)   As always throughout the article, we assume that $\maxdeg(G)<\infty$. Also, for $r\in \NN$  and Borel functions $\ell$ and $\ell'$ defined on $V$, we say that  $\ell'$ is an \emph{$r$-local function} of $(G,\ell)$ if the value of $\ell'$ on any vertex $x\in V$ depends only on the isomorphism type of the restriction of the labelled rooted digraph $(G,\ell,x)$ to the $r$-ball $N_G(x,r)$.
		
		We say that a finite partition $\pi$ of $V$ is \emph{Borel} if all elements of $\pi$ are Borel as subsets of~$V$. If $\pi$ is Borel then the map $S_{\pi}\colon V \to \pi$ is Borel.
		
		\begin{lemma}\label{lm:BorelPartition} Let $G=(V,\mathcal B,E)$ be a Borel digraph with all degrees bounded.  Then, for any $r\in \NN$ there exists a  Borel finite partition of $V$ which is $r$-sparse.
		\end{lemma}
		\begin{proof} 
			Consider  the unordered graph $\tilde G=(V,\tilde E)$ where 
			\beq\label{eq:TildeE}
			\tilde E:=( E\cup \{(y,x)\colon (x,y)\in E\})\setminus \{(x,x)\colon x\in V\},
			\eeq
			that is, we ignore the loops and the orientation of the arcs of $G$. The edge set $\tilde E$ is a Borel subset of $V^2$ as a Boolean
			combination of Borel sets. (For example, the diagonal $\{(x,x)\colon x\in V\}$ is a closed and thus Borel subset of $V^2$.) 
			By~\cite[Proposition~4.6]{KechrisSoleckiTodorcevic99} (or by \cite[Corollary 5.9]{Pikhurko21bcc}),
			the graph $\tilde G$, which is Borel and of bounded degree, admits a Borel $r$-sparse finite partition $\pi$. By definition,
			the partition $\pi$ is also $r$-sparse with respect to the digraph~$G$.
		\end{proof}

		Next, we are going to define when a local rule  $\mathbf R$ is Borel. 
		Let us fix a Borel $2$-sparse partition $\pi$ of~$V(G)$. The $2$-sparseness ensures that, for each $x\in V$, the labelling $S_\pi:V\to \pi$ is injective on $\Var(x)$. 
		Thus, in order to specify the rule $\mathbf R(x)$ at $x\in V$, it is enough to specify the set $\{\psi\in b^{\pi}\colon  (\psi\circ S_\pi)\restriction \Var(x)\in\mathbf R(x)\}$ since this set uniquely determines $\mathbf R(x)$.  This gives a natural (over-redundant) encoding of $\mathbf R$ by a function $\ell_{\mathbf R}:V\to \Pow(b^\pi)$. We say that the local rule $\mathbf R$  is \emph{Borel} if the function $\ell_{\mathbf R}$ is Borel. 
		
		Let us argue that this definition (whether $\mathbf R$ is Borel or not) does not depend on the choice of~$\pi$.
		Let $\pi'$ be another Borel $2$-sparse partition, with the corresponding encoding $\ell'_{\mathbf R}$. We can compute $\ell'_{\mathbf R}$ at any $x\in V$ from $\ell_{\mathbf R}(x)$  and the restriction of $(S_\pi,S_{\pi'}):V\to \pi\times \pi'$ to the $1$-ball $N_G(x,1)$ around $x$. This means that $\ell'_{\mathbf R}$ is a 1-local function of $(G,(\ell_{\mathbf R},S_\pi,S_{\pi'}))$. Now the fact that $\ell'_{\mathbf R}$ is Borel if $\ell_{\mathbf R}$ is Borel follows from the following general result.
		
		\begin{lemma}\label{lm:5.7} Let $G=(V,\mathcal B,E)$ be a Borel digraph with bounded degrees,  $r\in\NN$, and $\ell$ be a Borel vertex labelling of $G$ that assumes countably many values. Let $\ell'$ be an $r$-local function of $(G,\ell)$. Then the labelling $\ell'$ is Borel.
		\end{lemma}
		\begin{proof}[Proof sketch] This is a folklore result whose formal proof for unordered graphs can be found in~\cite[Lemma 5.17]{Pikhurko21bcc}. It can be applied to our digraph $G$ as follows. Fix a  Borel 2-sparse partition $\pi$ of $V(G)$. Let $\tilde G$ be the underlying unordered graph of $G$. Let $\tau:V(G)\to \pi\times 2^{\pi}\times 2^{\pi}$ be the labelling that maps a vertex $x$ to the index $S_\pi(x)$ of its part as well as to the subsets of $\pi$ corresponding to the out- and in-neighbourhoods of $x$ in $G$. The function $\tau$ can be shown to be Borel by using
the Lusin-Novikov Uniformisation Theorem (see e.g.\ the discussion in~\cite[Section~3.3]{Pikhurko21bcc}). Now, the lemma follows by applying~\cite[Lemma 5.17]{Pikhurko21bcc} to the $r$-local function on $(\tilde G,(\tau,\ell))$ that simulates~$\ell'$.\end{proof}

		
		\begin{lemma}\label{lem-borel-ref} Let $G=(V,\mathcal B,E)$ be a Borel digraph with $\maxdeg(G)<\infty$. 
			\part\label{it:a} For every Borel $A\subset V$ there is a maximal independent set $B\subset A$ which is a Borel subset of $V$.
			\part\label{it:b} If $\mathbf R$ is a Borel local rule on $G$ then, for every Borel labelling $f:V\to b$, the set $B_{\mathbf R}(f)\subset V$ of failed clauses is Borel.
			\trap
		\end{lemma}
		\begin{proof} 

			
			Recall that by an independent set in the digraph $G$ we mean an independent set in the underlying loopless Borel graph $\tilde G=(V,\tilde E)$, where $\tilde E$ is defined in~\eqref{eq:TildeE}. Thus Part (a) follows from the corresponding graph result of Kechris, Solecki and Todorcevic~\cite[Proposition~4.2]{KechrisSoleckiTodorcevic99} (which is Theorem~5.6 in~\cite{Pikhurko21bcc}).
			
			For Part (b), fix a Borel $2$-sparse partition $\pi$ of $V$ and use it to encode the local rule $\mathbf R$ by the Borel labelling $\ell_{\mathbf R}$ as above. For a clause $x\in V$, we can decide if it is failed by $f$ by knowing the restriction of $(G,(f,\ell_{\mathbf R},S_\pi),x)$ to $\Var(x)\subset N_G(x,1)$.
			Thus $B_{\mathbf R}(f)$ is Borel by the case $r=1$ of Lemma~\ref{lm:5.7}.
		\end{proof}
		
		A \emph{Borel Moser--Tardos tuple} is a tuple $(G, \mathbf R, \pi)$ 
		where $G$ is a Borel digraph with finite $\maxdeg(G)$, $\mathbf R$ is a Borel local rule on $G$,
		and $\pi$ is a Borel partition of~$V(G)$. One can show that the graph  $\Rel(G)$ is Borel (or, instead, operate with vertex-labellings on the original digraph~$G$ under appropriate encodings).
		Given $\rnd\in b^{\pi\times \NN}$ we 
		define the functions  $\MT^{0}_{\cal M, \rnd}$ and  $\MT^{i+1}_{\cal M, \rnd}$ for $i\in \NN$ by the formulas in 
		Subsection~\ref{mtalik} where we apply Part~\ref{it:a} of Lemma~\ref{lem-borel-ref} to find a maximal Borel $\Rel(G)$-independent subset of $B_{\mathbf R}(\MT^{i}_{\cal M, \rnd})$. Lemma~\ref{lm:5.7} gives by induction on $i$ that each function $\MT^{i}_{\cal M, \rnd}$ is Borel.
		
		A \emph{Borel colouring problem} is a pair $(G,\mathbf R)$ where $G$ is a Borel digraph and $\mathbf R$ is a Borel local rule on $G$.
		
		\begin{theorem}\label{thm-main-borel}
			Let $( G, \mathbf R)$ be a Borel colouring problem such that $G\in \subexp(\nn,\eps,d)$. Let $\De:=\maxdeg(\Rel(G))>0$ and let us assume that 
			$$
			\max_{x\in V(G)} \frac{b^{\eps d}\,|\mathbf R^c(x)|}{b^{|\Var(x)|}} < \frac{1}{\me\De}.
			$$
			Then there exists a satisfying colouring $f\in b^{V(G)}$ which is Borel.
		\end{theorem}
		\begin{proof}
			Let $\de>0$ be such that 
			$$
			\max_{x\in V(G)}\frac{b^{\eps d}\,|\mathbf R^c(x)|}{b^{|\Var(x)|}} \le \frac{1}{(\me\De)^{1+\delta}},
			$$
			and let $\pi$ be a $3\nn$-sparse Borel partition of $V(G)$ with finitely many parts, which exists by Lemma~\ref{lm:BorelPartition}. Let $\cal M$ be the Borel Moser--Tardos tuple $(G, \mathbf R, \pi)$. Now by Theorem~\ref{thm-main}, it holds that
			$\Pr_{\rnd}(\sup_{x\in V(G)} h^\infty(x)>m)$ tends to $0$ as $m\to\infty$. Thus for almost every \(\rnd\) we have that $\sup_{x\in V(G)} h^\infty(x)<\infty$. Let us fix one such instance of $\rnd$. Then we obtain a sequence of Borel colourings $\MT^0,\MT^1,\ldots$ of $V(G)$. Furthermore, for every $x\in V(G)$, the values $\MT^k(x)$ stabilise from some moment on; define $\MT^\infty(x):=\lim_{k\to\infty} \MT^k(x)$ to be this eventual colour. 
			Since $(\MT^\infty)^{-1}(a)=\cup_{n\in\NN}\cap_{k\ge n}(\MT^k)^{-1}(a)$ for all $a\in b$, 
			the function
			$\MT^\infty$ is Borel.
			
			It remains to verify that  $\MT^\infty$ is a satisfying colouring. Take any  $x\in V(G)$. Let $k\in \NN$ be such that 
			no clause sharing a variable with the clause $x$ is resampled after the $k$-th pass; such $k$ exists since these clauses involve finitely many (namely, at most $d^3$) variables in total.
			Thus, for every $l$ larger than $k$, we have that $N_{\Rel(G)}(x,1)\cap  \IB(\MT^l) = \emptyset$. By the maximality of $\IB(\MT^l)$ we have that $x\notin \B(\MT^l)$. Since,  $\MT^\infty$ coincides with $\MT^l$ on all variables of $x$ for $l>k$, we deduce that $x\notin \B(\MT^\infty)$, as desired. This finishes the proof of the theorem.
		\end{proof}
		
		\begin{myrem}\label{rem:noborel} 
			Let us note that the Borel LLL does not hold in full generality on graphs of exponential growth.  Marks~\cite{Marks16} constructed, for every $d\ge 3$, a $d$-regular acyclic Borel graph $F=(V,E)$ with a Borel $d$-edge-colouring $\ell:E\to d$ such that for every Borel vertex-colouring  $c:V\to d$ there is a \emph{bad pair}, namely  adjacent vertices $x,y\in V$ with $c(x)=c(y)=\ell(xy)$. Given $F$ and $\ell$, the problem of finding a colouring $c$ without bad pairs fits into the LLL framework (where the variables and the clauses correspond to the vertices and edges of $F$ respectively, with the clause $\mathbf R(xy)$ for $xy\in E(F)$ stating that the pair $xy$ is not bad). 
			Clearly, for a uniform $d$-colouring, the probability that a given pair is bad is $p:=1/d^2$ while the corresponding dependency graph (with self-loops) has maximum degree $\Delta:=2d-1$. Thus the LLL condition in~\eqref{eq:gap}, namely that $p<1/(\me\Delta)$, is satisfied if $d\ge 5$.  However, there is no Borel satisfying colouring.  
			Conley et al~\cite{ConleyJacksonMarksSewardTuckerdrob20} strengthened the above result of Marks~\cite{Marks16} by showing that the Borel graph $F$ can be additionally assumed to be hyperfinite. (A new and shorter proof of this is presented by Brandt et al~\cite{BrandtChangGrebikGrunauRozhonVidnyanszky24}.) Thus there are hyperfinite counterexamples to a Borel~LLL. 
		\end{myrem}
		
		\hide{
			\begin{myrem}  If we add the standard extra assumptions to Theorem~\ref{thm-main-borel}, namely that
				the underlying graph $G$ is a \emph{topological structured graph} (as in \cite[Definition 2.12]{Bernshteyn23i}, meaning that $V(G)$ is a zero-dimensional topological space such that, roughly speaking, the operation of mapping a vertex to its rooted radius-$R$ ball is continuous for every given $R\in\NN$),
				$V(G)$ has a continuous linear ordering, and the local rule $\mathbf R$ is a continuous function (when encoded as a map $V(G)\to \Pow(b^{d})$ using the induced ordering on each out-neighbourhood) then we can additionally require that the produced satisfying assignment $V(G)\to b$ is a continuous function. The proof is the same except we use the continuous analogues of Lemmas~\ref{lm:BorelPartition} and~\ref{lm:5.7} (which can be found in e.g.~\cite[Section 4.2]{Bernshteyn23i}).\end{myrem}
		}

\begin{ack}
	The authors thank Sebastian Brandt and Andrew Marks for useful discussions. Also, the authors are grateful to the anonymous referees for their careful reading and many helpful comments.
	ChatGPT 5.5 was used to proofread near-final versions of the paper.

	For the purpose of open access, the authors have applied a Creative Commons Attribution (CC-BY) licence to any Author Accepted Manuscript version arising from this
submission.
	\end{ack}

\begin{funding}
	Łukasz Grabowski was supported by ERC Starting Grant 805495.	
	András Máthé was supported by the Hungarian National Research, Development and Innovation Office -- NKFIH, 124749.
		Oleg Pikhurko was supported by ERC Advanced Grant 101020255 and Leverhulme Research Project Grant RPG-2018-424.
\end{funding}


		\bibliographystyle{plain}
		

	\end{document}